\documentclass[12pt,longbibliography]{article}

\usepackage[utf8]{inputenc}
\usepackage[T1]{fontenc}
\usepackage[a4paper, margin=2.7cm]{geometry}

\usepackage{amsmath, amsthm, amsfonts, amssymb}
\usepackage{mathrsfs}
\usepackage[colorlinks=true,linkcolor=blue,citecolor=blue,urlcolor=blue,breaklinks]{hyperref}
\usepackage{bbm} 
\usepackage{graphicx}

\usepackage{cite}
\usepackage{breakurl}
\usepackage{float}
\usepackage{xcolor}
\usepackage{url}

\def\N{\mathbb{N}}
\def\R{\mathbb{R}}

\def\d{\,\mathrm{d}}
\def\dv{\d v}

\def\p{\partial}

\def\:{\colon}

\newtheorem{thm}{Theorem}[section]

\newtheorem{lem}[thm]{Lemma}

\theoremstyle{definition}
\newtheorem{dfn}[thm]{Definition}
\theoremstyle{remark}
\newtheorem{rem}[thm]{Remark}
\theoremstyle{example}

\theoremstyle{conjecture}

\numberwithin{equation}{section}

\def\thetitle{Sequence of pseudo-equilibria describes the long-time behaviour of the NNLIF model with large delay}

\def\theauthor{María J.~Cáceres\footnote{\texttt{\href{mailto:caceresg@ugr.es}{caceresg@ugr.es}}} \& José~A.~Cañizo\footnote{\texttt{\href{mailto:canizo@ugr.es}{canizo@ugr.es}}} \& Alejandro Ramos-Lora\footnote{\texttt{\href{mailto:ramoslora@ugr.es}{ramoslora@ugr.es}}}\\Department of Applied Mathematics \& IMAG\\ University of Granada}

\hypersetup{pdftitle={\thetitle}}
\hypersetup{pdfauthor={\theauthor}}

\title{\thetitle}
\author{\theauthor}
\date{April 2024}

\begin{document}

\maketitle

\begin{abstract}
  There is a wide range of mathematical models that describe
  populations of large numbers of neurons.  In this article, we focus
  on nonlinear noisy leaky integrate and fire (NNLIF) models that
  describe neuronal activity at the level of the membrane potential.
  We introduce a sequence of novel states, which we call
  \emph{pseudo-equilibria}, and give evidence of their defining role
  in the behaviour of the NNLIF system when a significant synaptic
  delay is considered.  The advantage is that these states are
  determined solely by the system's parameters and are derived from a
  sequence of firing rates that result from solving a recurrence
  equation. We propose a new strategy to show convergence to an
  equilibrium for a weakly connected system with large transmission
  delay, based on following the sequence of pseudo-equilibria. Unlike
  direct entropy dissipation methods, this technique allows us to see
  how a large delay favours convergence. We present a detailed
  numerical study to support our results. This study helps understand,
  among other phenomena, the appearance of periodic solutions in
  strongly inhibitory networks.
\end{abstract}
\newpage
\tableofcontents

\section{Introduction}
\label{sec:intro}

Over the past few decades, a diverse range of research has emerged in
the realm of partial differential equation (PDE) systems to model
populations of large numbers of neurons. Depending on the variables
taken into account to describe the activity of neurons, several
families of PDE have been studied: Fokker-Planck equations including
voltage and conductance variables
\cite{Caikinetic,caceres2011numerical, perthame2013voltage,
  perthame2019derivation}; Fokker-Planck equations for uncoupled
neurons \cite{NKKRC10,NKKZRC10}; population density models of
integrate and fire neurons with jumps \cite{omurtag, henry13, henry12,
  dumont2020mean}; age structured equations (time elapsed models)
\cite{PPD,PPD2,pakdaman2014adaptation} which are derived as mean-field
limits of Hawkes processes
\cite{chevallier2015microscopic,chevallier2015mean}, McKean-Vlasov
equations \cite{acebron2004noisy, mischler2016kinetic}, which are the
mean-field equations of a large number of neurons described by the
Fitzhugh-Nagumo equation \cite{fitzhugh1961impulses}, etc.

\

This article is devoted to the \emph{nonlinear noisy leaky integrate and
fire (NNLIF) neuronal system}, which is one of the simplest PDE models
\cite{lapicque,sirovich,omurtag,brunel2000dynamics,BrGe,brunel1999fast}.
We focus on its mesoscopic/macroscopic description through mean-field
Fokker-Planck type equations
\cite{caceres2011analysis,carrillo2013classical,caceres2014beyond,
  caceres2017blow,
  caceres2018analysis,caceres2019global,hu2021structure,
  roux2021towards, ikeda2022theoretical}, although it has also been
studied at the microscopic level, using stochastic differential
equations (SDE)
\cite{delarue2015global,delarue2015particle,liu2022rigorous}.  Despite
their simplicity, and the extensive scientific output devoted to them,
fundamental questions remain about their long-term behaviour.  The aim
of this article is to contribute to the understanding of this
issue. Specifically, we give evidence that the behaviour of the NNLIF
model with large delay is determined by a simple discrete system,
which gives us a sequence of novel states, that we call
\emph{pseudo-equilibria sequence}. The advantage of the discrete
system lies in its simplicity: for instance, it allows for quick
simulations that provide accurate information about the NNLIF system.
We can rapidly study issues such as the long-term behavior of the system,
the estimated time to approach equilibrium, whether the system tends
toward a steady state,
the possible appearance of periodic solutions or \emph{plateau states}, etc., all in terms of the system parameters.

\medskip

We will first give a short introduction to the model, and then give our
main results. We present the microscopic and mesoscopic descriptions
of the NNLIF models.

\paragraph{Stochastic NNLIF neuron models.}

We are interested in describing the membrane potential, which is the
potential difference $V_i$ across the cell membrane of the $i$-th
neuron in a group of $\mathcal{N}$ interconnected neurons.  A widely
used model for this is the \emph{noisy leaky integrate-and-fire
  model}, which takes into account a simple mechanism for the neuron
to approach some natural \emph{resting potential}, $V_L$, and the
effect of the inputs from other neurons.  Thus, the evolution on time
of the membrane potential for neuron $i$ is given by
\begin{equation}
  C_m \frac{d V_i}{dt} = -g_L (V_i - V_L) + I_i(t),
  \label{start-sde}
\end{equation}
where $C_m$ is the capacitance of the membrane and $g_L$ is the leak
conductance.  Equation \eqref{start-sde} is stochastic, since the term
$I_i$ is the sum of all synaptic currents produced by the stochastic
spike trains of the $C_i$ neurons connected to neuron $i$:
\begin{equation*}
  I_i(t) := \sum_{j=1}^{C_i} \sum_{k} J_{ij} \delta (t - t_{j}^k - d).
\end{equation*}
The term $\delta (t - t_{j}^k - d)$ represents a Dirac delta (in
time), modelling the contribution of the $k$-th spike from the $j$-th
presynaptic neuron. It is modulated by the synaptic strength of the
$i$-$j$ connection, $J_{ij}$ (which may be positive or negative,
depending on whether the effect is excitatory or inhibitory,
respectively). The constant $d\ge 0$ is the synaptic delay---the time
it takes for the effect of a spike to be felt by other neurons.

In addition, this model imposes the condition that
whenever a neuron's membrane potential $V_i$ reaches a certain
\emph{firing potential} (or \emph{threshold potential} or
\emph{firing threshold}) $V_F$, it
discharges by sending a spike over the network, and then $V_i$ is
reset to a \emph{reset potential} value $V_R$. It is always assumed
that $V_L < V_R < V_F$.

The following are also common assumptions for this model: networks
with sparse random connectivity, i.e.  $C/\mathcal{N} \ll 1$;
strengths $|J_{ij}| \ll V_F$, i.e., small strengths compared to the
threshold potential; and that neurons spike according to stationary,
independent Poisson process, with a constant probability $\nu$ of
emitting a spike per unit time. Under these assumptions we can
consider the mean-field limit when $\mathcal{N}\to \infty$, see
\cite{brunel2000dynamics,brunel1999fast,renart2004mean,sirovich}. The
synaptic current is then approximated by a continuous-in-time
stochastic process of Ornstein-Uhlenbeck type with the same mean and
variance as the Poissonian spike-train process, so the synaptic
current of a typical neuron is
$$
I(t)dt\approx b \nu \d t+ a \d B,
$$
where 
$B$ is a Brownian process,  and $a > 0$ is the
strength of the noise.
The parameter $b$, called the {\em connectivity
  parameter}, is an \emph{average} connectivity strength and
encodes how excitatory or inhibitory the network is. The
case $b=0$ means that neurons are not connected to each other and
the system becomes linear.  Otherwise, if $b>0$ the network is
average-excitatory, and if $b<0$ the network is average-inhibitory.

Therefore, we obtain 
the stochastic differential equation (SDE) for a typical neuron
\begin{equation*}
  C_m \d V = -g_L (V - V_L) \d t + b\nu \d t + a \d B.
\end{equation*}
By appropriate translation and scaling one can remove the constants
$C_m$, $g_L$ and $V_L$ (see
\cite{caceres2011analysis,CR-L,delarue2015global,delarue2015particle}),
and obtain the SDE, which we write in the notation we will use
throughout the paper:
\begin{equation*}
  \d V = -(V + b N(t-d)) \d t + a \d B,
\end{equation*}
where $N = N(t)$ is the \emph{mean firing rate} of the network.

\paragraph{The NNLIF PDE model.}

The PDE satisfied by the probability distribution of $V$ from the
above equation is
\begin{equation}\label{eq:edp-delay}
  \p_t p(v, t)
  + \p_v\left[h\left(v,N(t-d)\right)p(v,t)\right]-a\p^2_vp(v,t)
  =\delta (v-V_R)N(t).
\end{equation}
It requires a given non-negative, integrable initial condition
$p(v,0)=p_0(v)\geq0$, regular enough for its associated firing rate to
exist:
$$N(t)=-a\p_vp_0(V_F) \;\text{for}\; t\in[-d,0].$$
This is the PDE we address in this paper, previously studied in many
references; see for example
\cite{brunel2000dynamics,delarue2015particle,delarue2015global,liu2022rigorous}
and the references therein. As described, it arises as the mean-field
limit of a large set of $\mathcal{N}$ identical neurons which are
connected to each other in a network, when $\mathcal{N} \to
+\infty$. This PDE provides the evolution in time $t\geq 0$ of the
probability density $p(v,t)\ge 0$ of finding neurons at voltage
$v\in\left(-\infty,V_F\right]$.  A neuron spikes when its membrane
voltage reaches the firing threshold value $V_F$, discharges
immediately after, and has its membrane potential set back to the
reset value $V_R$ ($V_R<V_F$), which is described by the right hand
side of the equation and the boundary condition $p(V_F,t)=0$.  In
addition, the model includes the delay $d$ in synaptic
transmission. In this paper we assume the diffusion coefficient $a>0$
to be constant, and we always take the drift coefficient $h$ to be
$h(v, N) := -v + bN$. This makes the system nonlinear since the firing
rate $N = N(t)$ must be computed as
\begin{equation}
  \label{eq:N-def}
  N(t):=-a\,\p_v p(V_F,t)\ge 0
\end{equation}
so that the total number of neurons $\int_{-\infty}^{V_F} p(v,t) \d v$
is conserved.
Due to our definition of $N$ in \eqref{eq:N-def}, and assuming
$p(-\infty, t)=0$, the solution to the related Cauchy problem
conserves the total number of neurons:
$$
\int_{-\infty}^{V_F}p(v,t)\d v=\int_{-\infty}^{V_F}p_0(v)\d v,
\quad \mbox{for} \ t\in [0, T),
$$
where $T>0$ is the maximal time of existence
\cite{carrillo2013classical,caceres2019global,roux2021towards}.

In general, NNLIF systems describe the activity of large numbers of
neurons in terms of the distribution of the membrane potential. The
PDE-based NNLIF family includes systems with different complexity
depending on the neurophysiological properties that are taken into
account \cite{C}.  In this article we always focus on the nonlinear
Fokker-Planck equation \eqref{eq:edp-delay}. For the sake of
simplicity, we always assume that $\int_{-\infty}^{V_F}p_0(v) \d v=1$
and take a diffusion coefficient $a=1$.

\paragraph{Stationary states.}

The probability \emph{stationary states}, \emph{steady states}, or
\emph{equilibria}, $p_\infty(v)$,  of system \eqref{eq:edp-delay} are
non negative solutions to
\begin{equation}
  \left\{
    \begin{array}{l}
      \p_v\left[\left(-v+bN_\infty\right)p_\infty(v)\right]-\p^2_v p_\infty(v)=\delta (v-V_R)N_\infty,
      \\
      N_\infty:=-\p_v p_\infty(V_F),
      \ \text{and} \
      \int_{-\infty}^{V_F}p_\infty(v) \d v=1,
    \end{array}
  \right.
  \label{eq: large-delta-stationary}
\end{equation}
 and are given by
\begin{equation}
  \label{eq:stationary-state}
  p_\infty(v) = N_\infty e^{-\frac{\left(v-bN_\infty\right)^2}{2}}
  \int_{\max(v, V_R)}^{V_F}e^{\frac{\left(w-bN_\infty\right)^2}{2}} \d w.
\end{equation}
The stationary firing rate $N_\infty$ is implicitly given by
the requirement of unit mass
($\int_{-\infty}^{V_F} p_\infty(v) \dv=1$), which is translated into
the condition $N_\infty I(N_\infty)=1$, where
\begin{equation}
  \label{eq: I}
  I(N):= \int_{-\infty}^{V_F}
  e^{-\frac{\left(v-bN\right)^2}{2}}
  \int_{\max(v, V_R)}^{V_F}e^{\frac{\left(w-bN\right)^2}{2}} \d w \d v.
\end{equation}
Every stationary state corresponds to a solution of the implicit
equation $N_\infty I(N_\infty)=1$ for $N_\infty$.  The number of
steady states depends on the connectivity parameter $b$
\cite{caceres2011analysis}: for $b\le 0$ (inhibitory case or the case where
neurons are not connected to each other) there is
only one, while for $b>0$ (excitatory case) there is more variety:
there is only one if $b$ is small; there are no steady states if $b$
is large; and there are two for intermediate values of $b$ (the
existence of at least two can be proved analytically; exactly two are
observed numerically).

\

\paragraph{Known results on the PDE.}

For the nonlinear Fokker-Planck Equation \eqref{eq:edp-delay} there is
a global in time solution if either $d>0$
\cite{caceres2019global,roux2021towards}, or if $d=0$ and $b\le 0$
(average-inhibitory and linear cases) \cite{carrillo2013classical}.
However, for the average-excitatory case ($b>0$), blow-up of the
solution may occur, and in this case the maximal time of existence is
given by the first time at which the firing rate $N(t)$ diverges
\cite{carrillo2013classical, caceres2011analysis}. This is known to
happen if there is no transmission delay and the initial condition is
sufficiently concentrated around the threshold potential $V_F$, or if
the connectivity parameter $b$ is large enough \cite{roux2021towards}.
The blow-up is avoided if some transmission delay or stochastic
discharge potential are considered \cite{caceres2014beyond}.  The
analogous criteria for existence and blow-up phenomena were studied in
\cite{delarue2015global}, for the associated microscopic system.
Moreover, the notion of physical solutions to the SDE was given in
\cite{delarue2015particle} and the authors proved physical solutions
are global on time, solutions continue after system synchronization,
this is after the blow-up phenomenon.  Understanding what ``physical
solution'' may mean for the Fokker-Planck equation is an open problem.
It was numerically analysed in \cite{CR-L}, and behaviour after the
explosion was studied in \cite{dou2022dilating} by a time change of
variable which dilates the blow-up time.

\

Let us give a brief review of the existing results on asymptotic
behaviour. Most of the literature on asymptotic behaviour of the model
\eqref{eq:edp-delay} is based on the entropy dissipation method (see
\cite{caceres2011analysis, caceres2014beyond, carrillo2014qualitative,
  caceres2017blow, caceres2019global}), considering a standard
relative entropy functional to estimate the distance between the
solution $p$ and the stationary solution $p_\infty$.
This strategy has given results only for weakly connected networks
(small values of the connectivity parameter $b$), that is, almost
linear systems.  Recently, approaches based on Doeblin \& Harris's
theorems in probability have been used in the study of the asymptotic
behaviour of various equations related to neural networks, for both
elapsed time \cite{canizo2019asymptotic} and integrate-and-fire models
\cite{dumont2020mean, salort2024convergence}. Again, both strategies
(entropy method and Harris-type theorems) seem to be suitable only for
nearly linear systems. Results on the long-term behaviour for general
connectivity have been recently given in \cite{local} by using a new
strategy based on the spectral gap properties of the linearised
equation. On the other hand, \cite{caceres2019global,roux2021towards}
proved that there are no periodic solutions if $b$ is large enough,
$V_F\leq0$ and a transmission delay is considered.  Recently in
\cite{ikeda2022theoretical} it was shown that periodic solutions can
appear in the average-inhibitory case if a large delay is considered
for an approximation of the integrate-and-fire model.  Moreover, for
the more realistic model (considering as different populations the
excitatory and inhibitory neurons, or neurons with refractory periods)
and the stochastic discharge potential model, the results in
\cite{caceres2018analysis,caceres2014beyond,hu2021structure,du2024synchronization,
  zhang2024spectral} numerically show periodic solutions.

\paragraph{Main ideas and the pseudo-equilibria sequence.}

Let us describe the main ideas in the present paper. We wish to
examine a discrete sequence determined by the parameters of the
nonlinear system \eqref{eq:edp-delay}.  For a given $N \geq 0$ we
define the \emph{pseudo-equilibrium} associated to $N$ as
\begin{equation}
  \label{eq:profile-plateau}
  p_{\tiny\mbox{pseudo}}(v):= \tilde{N}e^{-\frac{\left(v-bN\right)^2}{2}}
  \int_{\max(v, V_R)}^{V_F}e^{\frac{\left(w-bN\right)^2}{2}} \d w,
  \quad \mbox{with} \quad \tilde{N}=\frac{1}{I(N)}.
\end{equation}
This is just an equilibrium for the linear version of equation
(\ref{eq:edp-delay}), where the dependence of $h$ on $N(t-d)$ is
``frozen'' at a certain value $N$. The constant $\tilde{N}$ guarantees
unit mass: $\int_{-\infty}^{V_F} p_{\tiny\mbox{pseudo}}(v) \dv =
1$. The profile $p_{\tiny\mbox{pseudo}}$ is an equilibrium of the
NNLIF system (\ref{eq:edp-delay}) if and only if and only if
$\tilde{N} = N$, that is, if and only if $N I(N)=1$.

Our guiding idea is that for a large delay $d$, the NNLIF system
(\ref{eq:edp-delay}) behaves almost as a linear system in time
intervals of length $d$. Assume that we start with a constant initial
condition on $[-d,0]$, which has a firing rate $N_{0,\infty}$. Then
the NNLIF system is exactly linear on $[0,d]$; $d$ is large, we
expect the solution at $t=d$ to be close to the linear equilibrium
associated to a \emph{fixed} firing rate $N_{0,\infty}$, that is, the
pseudo-equilibrium associated to a certain firing rate
$N_{1,\infty}$. This corresponds to taking $N=N_{0,\infty}$ and
$\tilde{N}=N_{1,\infty}$ in \eqref{eq:profile-plateau}.
On the time interval $[d,2d]$, if on $[0,d]$ the
firing rate has remained close to $N_{1,\infty}$, we expect the
solution to again behave linearly, and approach another
pseudo-equilibrium with firing rate $N_{2,\infty}$. Iterating this we
define a sequence of pseudo-equilibria determined solely by the
initial value $N_{0,\infty}$ (and the given parameters $V_R$, $V_F$,
$b$). We are  able to rigorously prove that this sequence
represents the behaviour of the system with long delays  for the
case of small $b$ (see Section \ref{sec:analytical-results}),
but we are unable to do so for general values of $b$.
However, we present strong numerical evidence that this sequence does represent
the overall behaviour of the system (cf. Section
\ref{sec:numerical}). In particular, this sequence shows periodic
behaviour  for $b$ negative enough;
converges to the stable equilibrium of (\ref{eq:edp-delay}) in the
region where solutions (\ref{eq:edp-delay}) do the same; and approaches
a \emph{plateau state} when solutions to (\ref{eq:edp-delay}) are
expected to do this.

Let us describe this plateau state more closely, since it is also an
important evidence for the usefulness of the pseudo-equilibria
sequence. Recently, in \cite{CR-L} we nume\-ri\-cally observed the
formation of a new profile, which des\-cribes networks with uniformly
distributed membrane potentials between the reset value $V_R$ and the
threshold value $V_F$. We called these states ``plateau''
distributions.  These profiles appear in two situations: 1) strongly
connected systems (high connectivity parameter $b$) with synaptic
delay $d>0$, and 2) systems with $b=V_F-V_R$, without delay ($d=0$) or
with very small transmission delay value.  Numerically we observe that
plateau states appear to be related to pseudo-equilibria in the
following way: Figure \ref{fig:comparison_profiles-plateau} shows the
comparison of the plateau distributions found for $b=1.5$ (left plot)
and $b=2.2$ (right plot), with delay transmission, and several
pseudo-equilibria.  There is a high coincidence between the plateau
distributions and the profiles \eqref{eq:profile-plateau} as $N$
increases.  Furthermore, on the way to the plateau state, the dynamics
of the particle system appears to pass through pseudo-equilibria of
the form \eqref{eq:profile-plateau} as $N$ increases. For $b=1.5$ the
system \eqref{eq:edp-delay} has two steady states, and it can evolve
towards a plateau distribution (shown on the left plot) if the initial
condition is far from the stationary state with lowest firing rate.
For $b=2.2$ there are no stationary solutions and the system evolves
towards a plateau distribution.

\begin{figure}[H]
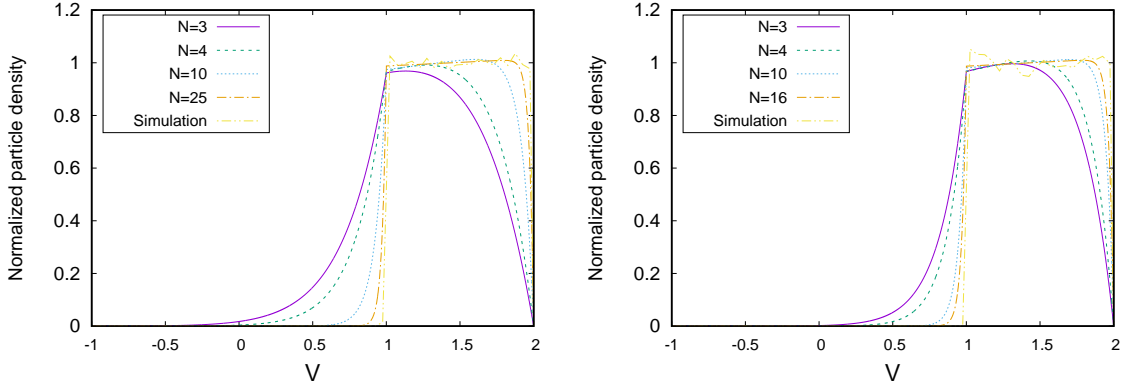

  \begin{center}
    \begin{minipage}[c]{0.48\linewidth}
      \begin{center}
	\includegraphics[width=\textwidth]{plateau_1.eps}
      \end{center}
    \end{minipage}
    \begin{minipage}[c]{0.48\linewidth}
      \begin{center}
        \includegraphics[width=\textwidth]{plateau_2.eps}
      \end{center}
    \end{minipage}
  \end{center}
  \caption{{\bf Comparison  of ``plateau'' distributions with
      profiles \eqref{eq:profile-plateau}}.
      Graphs taken from \cite{CR-L}.
    {\emph Left:} $b=1.5$.
    {\emph Right:}  $b=2.2$.}
  \label{fig:comparison_profiles-plateau}
\end{figure}

This behaviour suggests that we may think of the pseudo-equilibria
sequence as a tool to analyse the asymptotic behaviour of the system
when a large synaptic delay is taken into account. The aims of this
work are twofold:
\begin{enumerate}
\item to analyse the sequence of pseudo-equilibria determined by
  the associated discrete system, and
  
\item to establish a connection between the long-term behaviour of the
  pseudo-equilibria sequence and the solutions of the nonlinear system
  \eqref{eq:edp-delay}.
\end{enumerate}

\paragraph{Organisation of the paper.}

In Section \ref{sec:pseudo-equilibria-sequence} we define and analyse
the firing rate and the pseudo-equilibria sequences. In Section
\ref{sec:analytical-results} we prove the exponential convergence to
equilibrium for the nonlinear system \eqref{eq:edp-delay} for small
connectivity parameter $ b $ and large delay $ d $, by ``following''
the sequence of pseudo-equilibria, under some technical assumptions on
the linear systems. A similar result was previously established
through the use of the entropy method, and we present  here
different strategy which considers the relative entropy to the
sequence of pseudo-equilibria.  This approach looks promising to study
a broader range of phenomena (such as periodic solutions or the
approach of the plateau state), and served as a catalyst for the main
ideas in our paper \cite{local}, where we prove local convergence
results depending on the values of the transmission delay and the
connectivity parameter.  In Section \ref{sec:numerical} we provide the
numerical evidence to support the connection between the sequence of
pseudo-equilibria studied in Section
\ref{sec:pseudo-equilibria-sequence} with the nonlinear system.  This
numerical work also suggests several conjectures on the possible
extension of the results of Section \ref{sec:analytical-results} to
cases with an arbitrary value of the parameter $b$. Section
\ref{sec:conclusions} is devoted to conclusions, discussion and
possible extensions of our work.  We also include an appendix with
some of the technical results needed to prove Theorem
\ref{th:relations-Nb-negative}.
\section{Firing rate and pseudo-equilibria sequences}
\label{sec:pseudo-equilibria-sequence}

This section deals with the study of sequences of firing rates,
$\left\{N_{k,\infty}\right\}_{k\ge 0}$, and pseudo-equilibria,
$\left\{p_{k,\infty}\right\}_{k\ge 0}$. Let us first give their
definition:

\begin{dfn}[Firing rate sequence]
  \label{def: firing-seq}
  Given $b \in \R$, $V_R<V_F \in \R$, and $I$ the function
  \eqref{eq: I}, the \emph{firing rate sequence}
  $\left\{N_{k,\infty}\right\}_{k\ge 0}$ associated to an initial
  firing rate $N_{0,\infty} > 0$ is recursively defined by
  \begin{equation} \label{eq:sequence-N}
    N_{k+1,\infty}:=\frac{1}{I(N_{k,\infty})} \qquad k= 0,1,2,\dots
  \end{equation}
\end{dfn}
\noindent
This sequence is well defined since $I(N)>0$ for all $N >
0$. Associated to a firing rate sequence $\{N_{k,\infty}\}_{k\ge 0}$
we define the following pseudo-equilibria sequence:
\begin{dfn}[Pseudo-equilibria sequence]
  Given a firing rate sequence $\left\{N_{k,\infty}\right\}_{k\ge 0}$
  its associated \emph{pseudo-equilibria sequence} is the sequence
  \begin{equation}\label{eq:sequence-p}
    p_{k,\infty}(v) = N_{k,\infty}e^{-\frac{\left(v-bN_{k-1,\infty} \right)^2}{2}}
    \int_{\max(v, V_R)}^{V_F}e^{\frac{\left(w-bN_{k-1,\infty}
        \right)^2}{2}}\d w,
    \quad k=1,2,\ldots
  \end{equation}
  \label{def: pseudo-seq}
\end{dfn}
The behaviour of these sequences depends only on the values of the
connectivity parameter $b$, the diffusion coefficient $a$ (which we
consider to be 1), the reset potential $V_R$ and the threshold
potential $V_F$. Their definition is completely detached from the
dynamics of the nonlinear system \eqref{eq:edp-delay}.

In the following theorems we show the monotonicity of the firing rate
sequence $\left\{N_{k,\infty}\right\}_{k\ge 0}$, in terms of the
number of solutions to the implicit equation
\begin{equation}\label{eq: implicit}
	NI(N)=1, \ \mbox{with} \ 0\le N.
\end{equation}
We have denoted the unknown in Equation \eqref{eq: implicit} by $N$
and hope that it will not cause any confusion with the firing rate of
the nonlinear system \eqref{eq:edp-delay}. In this section we never
consider the system \eqref{eq:edp-delay}, since we are only concerned
with the behaviour of the firing rate sequence.

The number of solutions to Equation \eqref{eq: implicit} was studied
in \cite{caceres2011analysis}. Analytically it was shown that for
$b<0$ or for a small positive value of $b$ there is only one solution,
there is no solution if $b>0$ is large, and there are at least two for
intermediate positive values of $b$.  The proof is based on the study
of the monotone function $1/I(N)$ (increasing if $b>0$ and decreasing
if $b<0$). However, rigorously proving the exact number of solutions
in the excitatory case ($b>0$) is not easy, due to the complexity of
$I$. Numerically we observe that for positive $b$ there are a maximum
of two solutions, and at the value of $b$ where the equation switches
from having two solutions to none the function $1/I(N)$ has slope
equal to 1 and positive second derivative (see right plot of Figure
\ref{fig: function_I}). These are the scenarios we consider in the
following theorem for the case $b>0$:

\begin{thm}[Monotonicity of the firing rate sequence
  $\left\{N_{k,\infty}\right\}_{k\ge 0}$ for excitatory networks]
  \label{th:relations-Nb-positive}
  Let us consider $0<b$ and a firing rate sequence
  $\left\{N_{k,\infty}\right\}_{k\ge 0}$ (as given by Definition
  \ref{def: firing-seq}).
  \begin{enumerate}
  \item If Equation \eqref{eq: implicit} has a unique solution $N^*$
    then:
    \begin{enumerate}
    \item If $N \mapsto 1/I(N)$ crosses the diagonal $N \mapsto N$ (in
      the sense that $N<1/I(N)$ for $N<N^*$ and $1/I(N)<N$ for
      $N^*<N$), then:
      \begin{itemize}
      \item For $N_{0,\infty} \geq N^*$, the sequence
        $\left\{N_{k,\infty}\right\}_{k\ge 0}$ is decreasing
        and tends to $N^*$.
      \item For $ N_{0,\infty}\le N^* $, the sequence 
        $\left\{N_{k,\infty}\right\}_{k\ge 0}$ is increasing
        and tends to $N^*$.
      \end{itemize}
    \item If $ N\mapsto1/I(N) $ stays on one side of the diagonal
      $ N\mapsto N $ (in the sense that $N<1/I(N)$ for all
      $N\neq N^*$), then:
      \begin{itemize}
      \item For $N_{0,\infty} > N^*$, the sequence
        $\left\{N_{k,\infty}\right\}_{k\ge 0}$ diverges.
      \item For $ N_{0,\infty}\le N^* $, the sequence
        $\left\{N_{k,\infty}\right\}_{k\ge 0}$ is increasing and tends
        to $N^*$.
      \end{itemize}
    \end{enumerate}
  \item If Equation \eqref{eq: implicit} has no solutions, then
    $\left\{N_{k,\infty}\right\}_{k\ge 0}$ diverges.
    
  \item If Equation \eqref{eq: implicit} crosses the diagonal
    $N \mapsto N$ at exactly two points $N_1^* < N_2^*$ (in the sense
    that $1/I(N) > N$ for $N < N_1^*$, $1/I(N) < N$ for
    $N_1^* < N < N_2^*$, and $1/I(N) > N$ for $N > N_2^*$) then:
    \begin{itemize}
    \item If $ N_{0,\infty}\le N_1^* $ then
      $\left\{N_{k,\infty}\right\}_{k\ge 0}$ is an increasing sequence
      which tends to $N_1^*$.
    \item If $ N_1^*\le N_{0,\infty}\le N_2^* $
      then 
      $\left\{N_{k,\infty}\right\}_{k\ge 0}$ is a decreasing
      sequence which tends to $N_1^*$.
    \item If $N_2^* < N_{0,\infty} $ then
      $\left\{N_{k,\infty}\right\}_{k\ge 0}$ diverges.
    \end{itemize}
    
  \end{enumerate}
\end{thm}

\begin{proof}
  The function $I$ (see \eqref{eq: I}) is a $C^\infty(0,\infty)$
  function, which was studied in the proof of
  \cite[Theorem 3.1]{caceres2011analysis},
  and  can be  rewritten as
  $
  I(N)=\int_0^\infty e^{-s^2/2} e^{-sbN}
  \frac{e^{sV_F}-e^{sV_R}}{s}
  \ \d s.
  $
  Its $k$-th order derivative is
  \begin{equation}
    \label{eq: I-derivate}
    I^{(k)}(N)=(-b)^k
    \int_0^\infty  s^{k-1}e^{-s^2/2} e^{-sbN}
    \left(e^{sV_F}-e^{sV_R}\right)
    \ \d s,
  \end{equation}
  and  for $b$ positive:
  \begin{itemize}
  \item  $\frac{1}{I(N)}$ is an increasing function.
  \item $\frac{1}{I(0)}<\infty$.
  \item $\lim_{N\to \infty}\frac{1}{I(N)}=\infty$.
  \end{itemize}
  The firing rate sequence is a solution to the recursive equation
  \begin{equation}
    \label{eq: difference-equation-exc}
    N_{k+1,\infty}=f(N_{k,\infty}),\quad 
    f:\R_0^+ \to \R^+,
    \ f(x):=\frac{1}{I(x)},
  \end{equation}
  where $ \R^+:=\{x\in\R \mid x>0\} $,
  $ \R_0^+:=\{x\in\R \mid x\ge0\} $, and we have seen that $f$ is an
  increasing function. The behaviour of this type of discrete system
  is well known: its solutions are monotonic, and either diverge or
  converge to an equilibrium of the system.  With the hypotheses of
  the theorem we have the following possibilities, since $f(0)>0$:
  \begin{itemize}
  \item In case 1(a),
    \begin{itemize}
    \item $N<f(N)$ for $N<N^*$.
    \item $f(N)<N$ for $N^*<N$.
    \end{itemize}
  \item In case 1(b), $N<f(N)$ for $N\neq N^*$.

  \item In case 2 we have $N<f(N)$ for all $N\ge 0$, since $0<f(0)$.

  \item In case 3, 
    \begin{itemize}
    \item $N < f(N)$ for
      $N\in [0,N_1^*)\cup (N_2^*,+\infty)$.
    \item $f(N)<N$ for $N\in (N_1^*,N_2^*)$.
    \item $N < f(N)$ for $N \in (N_2^*,+\infty)$.
    \end{itemize}
  \end{itemize}
  In cases 1 and 3, this is just part of the hypotheses of the result;
  in case 2 this is a consequence of $f(0) > 0$, and the facts that
  $f$ is continuous and (by assumption in case 2) does not touch the
  diagonal.
  \begin{enumerate}
  \item If the initial condition is taken in an interval in which
    $N<f(N)$ then the sequence is increasing, because
    $N_{k,\infty}<f(N_{k,\infty})=N_{k+1,\infty}$.
  \item If the initial condition is taken in an interval in which
    $f(N)<N$ then the sequence is decreasing,
    because    $N_{k+1,\infty}=f(N_{k,\infty})<N_{k,\infty}$.
  \end{enumerate}
  Whenever the sequence is decreasing, it is bounded below by a
  constant solution of Equation \eqref{eq: difference-equation-exc}
  and therefore converges to it (because there are no more equilibria
  in that interval).
  If the sequence is increasing, it either converges to an
  equilibrium, if bounded, or diverges.
\end{proof}

\begin{rem}\label{rem:derivate-bpositive}
  As a consequence of the convergence of the firing rate sequences,
  given in Theorem \ref{th:relations-Nb-positive}, it follows:
 \begin{enumerate}
 \item If $ N^* $ is the only solution to Equation
   \eqref{eq: implicit}, then:
   \begin{enumerate}
   \item In case 1.(a) of Theorem \ref{th:relations-Nb-positive},
     $\frac{d}{dN}\frac{1}{I(N)}|_{N^*} \leq 1$.
   	\item In case 1.(b) of Theorem \ref{th:relations-Nb-positive}, $\frac{d}{dN}\frac{1}{I(N)}|_{N^*}= 1$.
   \end{enumerate}
   This behaviour is shown in plot on the right of Figure \ref{fig:
     function_I} with $ b=0$ or $ b=0.5 $ for the first case and
   $ b=2.1 $ the second case.
\item If Equation \eqref{eq: implicit} has two solutions: $ N_1^* $
  and $ N_2^* $ ($N_1^* <N_2^*$), then:
  \begin{itemize}
  \item   $0\le \frac{d}{dN}\frac{1}{I(N)}|_{N_1^*}\le 1$.
  \item   $1 \leq \frac{d}{dN}\frac{1}{I(N)}|_{N_2^*}$.
  \end{itemize}
  This behaviour is shown in the plot on the right of Figure \ref{fig:
    function_I} with $ b=1.5$.
  
\end{enumerate}
\end{rem}
We analyse the inhibitory case ($b<0$) in the following theorem, where
two different behaviours appear although in this case Equation
\eqref{eq: implicit} has only one solution.

\begin{thm}[Monotonicity of the firing rate sequence
  $\left\{N_{k,\infty}\right\}_{k\ge 0}$ for inhibitory networks]
  \label{th:relations-Nb-negative}
  Assume $b<0$ and consider a firing rate sequence
  $\left\{N_{k,\infty}\right\}_{k\ge 0}$. Then there exists a value
  $b^*< 0$ of the connectivity parameter $b$ such that:
  \begin{itemize}
  \item If $b^*< b<0$, the unique
    solution $N^*$ to Equation \eqref{eq: implicit} is asymptotically
    stable regarding firing rate sequences.
  \item    If $b< b^*$, there exist two values $N^-$, $N^+$,
    $0\le \min(N^-, N^+)<N^* <\max(N^-, N^+)$, such that  the sequence
    $\left\{N_{k,\infty}\right\}_{k\ge 0}$  tends to the 2-cycle $\{ N^-,N^+\}$.
  \end{itemize}
\end{thm}

\begin{proof}
  \textbf{Step 1: The sequence of firing rates is either convergent or
    asymptotically 2-periodic.} The proof, as in Theorem
  \ref{th:relations-Nb-positive}, is based on the study of the
  recursive equation:
  \begin{equation}
    \label{eq: difference-equation-inh}
    N_{k+1,\infty}=f(N_{k,\infty}),\quad 
    \ f(x):=\frac{1}{I(x)},
  \end{equation}
  where the function $I$, defined in \eqref{eq: I}, is a
  $C^\infty(0,\infty)$ increasing function if $b<0$ (see \eqref{eq:
    I-derivate} with $k=1$).  Therefore, for the inhibitory case,
  $f:[0, +\infty) \to \left[0,\frac{1}{I(0)} \right]$ since the
    function $1/I$ is decreasing, and, with the possible exception of
    the initial datum, $N_{0,\infty}$, all the other terms in the
    sequence fall within the interval
    $\left[0, \frac{1}{I(0)}\right]$. Moreover, in this case, Equation
    \eqref{eq: implicit} has an unique solution, $N^*$, because
    $1/I(N)$ is a decreasing continuous function which tends towards
    0, and $0<1/I(0)$.

    The behaviour of this type of discrete system is easy to study,
    because $f$ is a decreasing function and the solutions are
    bounded; they tend to the unique equilibrium or towards a 2-cycle.
    The proof of this result is as follows: denote by $N^*$ the
    equilibrium of \eqref{eq: difference-equation-inh}, and assume
    that $N_{0,\infty}\neq N_{2,\infty}$ (otherwise, the solution
    would be exactly a 2 cycle, or constantly equal to $N^*$ if
    $N_{0,\infty} = N_{2,\infty} = N^*$). Thus, using that $f$ is a
    decreasing function and $F:=f\circ f$ is increasing:
  \begin{itemize}
  \item If $N_{0,\infty} < N_{2,\infty}$, then
    $N_{3,\infty}=f(N_{2,\infty}) \le f(N_{0,\infty}) = N_{1,\infty}$
    and $N_{2,\infty}=F(N_{0,\infty})\le F(N_{2,\infty}) = N_{4,\infty}$.
    Therefore, by induction we prove that
    $\{N_{2k,\infty}\}_{k\ge 0}$ is an increasing sequence and
    $\{N_{2k+1,\infty}\}_{k\ge 0}$ is a decreasing sequence.
  \item If $N_{0,\infty} > N_{2,\infty}$, then
    $N_{3,\infty}=f(N_{2,\infty})\ge N_{1,\infty}=f(N_{0,\infty})$
    and $N_{2,\infty}=F(N_{0,\infty})\ge N_{4,\infty}=F(N_{1,\infty})$.
    Therefore, by induction we prove that
    $\{N_{2k,\infty}\}_{k\ge 0}$ is a decreasing sequence and
    $\{N_{2k+1,\infty}\}_{k\ge 0}$ is an increasing sequence.    
  \end{itemize}
  Moreover:
  \begin{itemize}
  \item If $N_{0,\infty}\le N^*$, then
    $N^*=f(N^*)\le f(N_{0,\infty})=N_{1,\infty}$ and
    $N_{2,\infty}=f(N_{1,\infty})\le f(N^*)=N^*$.  Therefore, by
    induction we prove that
    $0\le N_{2k,\infty}\le N^*\le N_{2k+1,\infty}\le \frac{1}{I(0)}$
    for $k=0,1,2,\ldots$.
  \item If $ N^*\le N_{0,\infty}$, then
    $f(N_{0,\infty})=N_{1,\infty}\le N^*=f(N^*)$ and
    $ f(N^*)=N^* \le N_{2,\infty}=f(N_{1,\infty})$.  Therefore, by
    induction we prove that
    $0\le N_{2k+1,\infty}\le N^*\le N_{2k,\infty}\le \frac{1}{I(0)}$
    for $k=0,1,2,\ldots$.
  \end{itemize}
  So that in either case, both sequences, $\{N_{2k,\infty}\}_{k\ge 0}$ and
  $\{N_{2k+1,\infty}\}_{k\ge 0}$, are monotonic and bounded, hence convergent.
  Thus,  considering $$N^-:=\lim_{k\to \infty} N_{2k+1,\infty} \;\text{and}\;
  N^+:=\lim_{k\to \infty} N_{2k,\infty},$$ we obtain, since $f$ is
  a continuous function, that $f(N^+)=N^-$ and there
  are two possibilities:
  \begin{enumerate}
  \item $N^-=N^+=N^*$, so that the system tends towards the equilibrium.
  \item $N^-\neq N^+$, so that the system tends towards the 2-cycle
    $\{N^-,N^+\}$, where $\min(N^-, N^+) < N^* < \max(N^-, N^+)$.
  \end{enumerate}
  
  \
  
  \textbf{Step 2: Determination of $b_*$.} The remainder of the proof
  concentrates on determining $b^*$, which must be the largest value
  at which $N^*$ is asymptotically stable for our discrete iteration.
  When we need to emphasize the dependence of $I$ on $b$ we will use
  the notation $I(b,N)$ (and $f(b,N) := 1/I(b,N)$). We note that $I$
  is a decreasing function of $b$, since
  $$\partial_b I(b,N)=-N\int_0^\infty
  e^{-s^2/2}e^{-sbN}\left(e^{sV_F}-e^{sV_R}\right) \d s,$$ and
  consequently $1/I$ is increasing as a function of $ b $ (see Figure
  \ref{fig: function_I}).  Therefore, if $b_1<b_2<0$, then
  $\frac{1}{I(b_1,N)}<\frac{1}{I(b_2,N)}$ and the solution to Equation
  \eqref{eq: implicit} for $b=b_1$, is less than the solution for
  $b=b_2$, that is, $N^*_{b_1}<N^*_{b_2}$, with $N_b^*$ denoting the
  (unique) solution to \eqref{eq: implicit} for a given $b \leq 0$.
  Alternatively, one can prove that $N_b^*$ is increasing in $b$ by
  deriving implicitly in the equation $N_b^*I(b,N_b^*)=1$ and
  observing that the derivative is positive, as follows:
  $$\frac{d N_b^*}{db} = \frac{-N_b^* \partial_bI(b,N_b^*)}{I(b,N_b^*)+
    N_b^*\partial_NI(b,N_b^*)}\ge 0.$$
  Consequently, $N^*_b$ is an increasing function bounded from below
  by 0, so it has a limit when $b$ tends to $-\infty$.  On the other
  hand, $ \overline{f}(N):=\lim_{b\to-\infty}f(b,N) $ loses continuity
  at $ N=0 $, because $\overline{f}(0)>0$, while it vanishes for
  $0<N<\overline{f}(0)$. Therefore,
  $$\overline{N}^*:=\lim_{b\to -\infty}N^*_b=0.$$
  This is because if the limit was positive, $\overline{N}^*>0$, then
  using the explicit expression of $I$ we get
  $\lim_{b\to -\infty}f(b,N_b^*)=0$, and therefore
  $$
  0=\lim_{b\to -\infty}f(b,N_b^*)=\lim_{b\to -\infty}N_b^*= \overline{N}^*>0,
  $$
  which is a contradiction.
 
  (This was also proven in \cite[Lemma 2.2]{ikeda2022theoretical}).
  Thus, in a certain sense, the system ``loses'' its equilibrium when
  $b$ tends to $-\infty$, due to the loss of continuity of the
  function $1/I$ in the limit.
	
  To determine whether or not the equilibrium of Equation \eqref{eq:
    difference-equation-inh}, $N^*_b$, is asymp\-to\-ti\-cally stable,
  we check whether $\vert f'(N^*_b)\vert$ is less than or greater than 1.
  With this aim, we
  define
  $$g(b):=f'(N^*_b)=\frac{d}{dN}\frac{1}{I(N)}\Big|_{N^*_b}=
  \frac{-\p_NI(b,N^*_b)}{I(b,N^*_b)^2} =-{N_b^*}^2\p_NI(b,N^*_b),$$
  which is a continuous increasing function in $(-\infty,0]$, due to
  the positivity of $ g'(b) $, as proved in Lemma
  \ref{lem:monotonicity_g}.  We note that
  $\lim_{b\to -\infty}g(b)=-\infty$ and $g(0)= f'(N^*_0)=0$, thus
  there exists $b^*$ such that $g(b^*)=-1$.  Then, if $b<b^*$, thus
  $g(b)<-1$, which means that $N_b^*$ is unstable for the discrete
  iteration, while for $b>b^*$, $g(b)>-1$ and $N_b^*$ is stable.  And
  as we proved above, if the equilibrium is unstable, the system
  tends to a 2-cycle.
\end{proof}

\begin{rem}
  We do not know how to prove analytically some further details in
  Theorem \ref{th:relations-Nb-negative}, due to the difficulty of
  working with the function $I$: the global stability of the equilibrium
  in the case $b^*\le b<0$, 
  and the uniqueness of the two cycles, in
  the case $b<b^*$. However, we can check both numerically.  Figure
  \ref{fig:cycles} shows the function $ F(N)=f\circ f (N) $, for cases
  with $ b<b^* $. $F$ has only $ 3 $ fixed points, i.e., the
  equilibrium $ N^* $ and cycle $ \left\{{N^-,N^+}\right\} $
  corresponding to each value of $ b $. Moreover, the 2-cycles appear
  as a bifurcation of the equilibrium, which is asymptotically stable
  for $b=b^*$, as this value is the starting point for the formation
  of 2-cycles.  We also note that the 2-cycle
  $ \left\{{N^-,N^+}\right\} $ becomes
  $\left\{0,\frac{1}{I(0)}\right\}$, approximately
  $ \left\{0,0.12\right\} $, when $b$ tends to $-\infty$, as $N_{b}^*$
  tends to 0. For $ b>b^* $ we observe an unique fixed point for
  $ F(N) $, corresponding with the equilibrium $ N^* $, so it is
  globally stable.
\end{rem}
\begin{rem}\label{rem:derivate-bnegative}
  As a consequence of the convergence of the firing rate sequences,
  given in Theorem \ref{th:relations-Nb-negative}, it follows:
  \begin{enumerate}
  \item If $b_*<b<0$, then $-1\le\frac{d}{dN}\frac{1}{I(N)}|_{N^*}<0$,
    for  $N^*$ the only solution to Equation \eqref{eq: implicit}.
  \item If $b<b_*$, then   $\frac{d}{dN}\frac{1}{I(N)}|_{N^*}\le -1$,
    for  $N^*$ the only solution to Equation \eqref{eq: implicit}. 
  \end{enumerate}
  We have numerically estimated $b^* \approx -9.4$ (see Figure
  \ref{fig:cycles}). The plot on the left of Figure \ref{fig:
    function_I} shows the function $\frac{1}{I(N)}$ for $ b<-9.4 $ and
  for $ b>-9.4 $, where we can observe its slope at $N^*$.
\end{rem}

\begin{figure}[h]
	\begin{center}
		\begin{minipage}[c]{0.49\linewidth}
			\begin{center}
				\includegraphics[width=\textwidth]{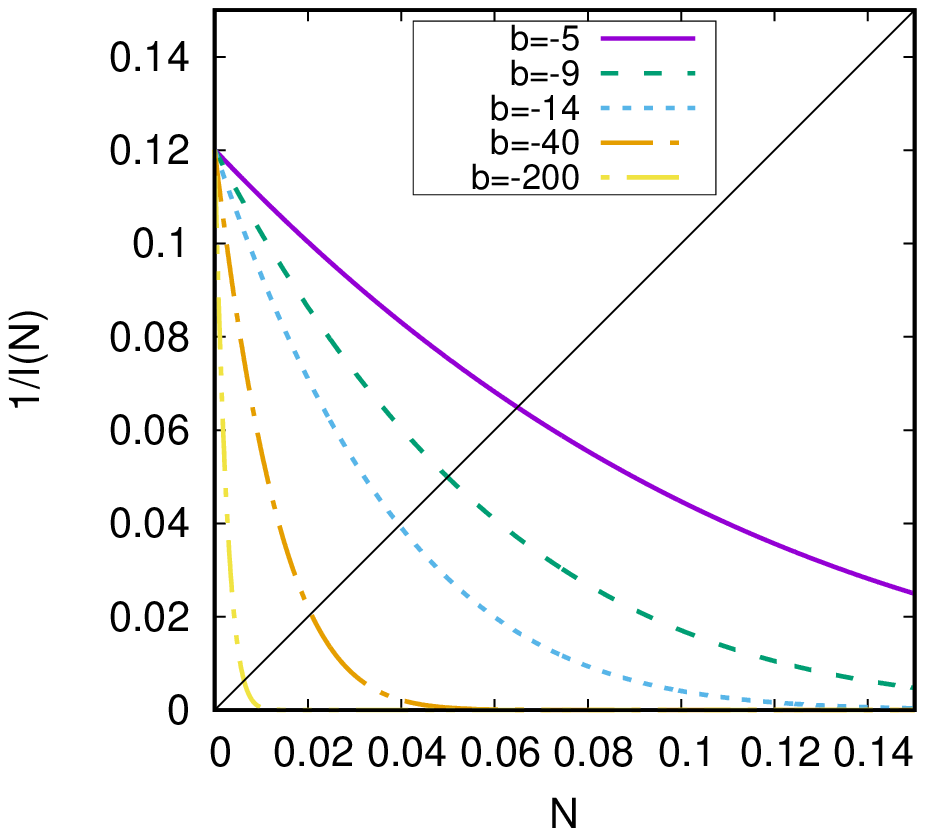}
			\end{center}
		\end{minipage}
		\begin{minipage}[c]{0.49\linewidth}
			\begin{center}
				\includegraphics[width=\textwidth]{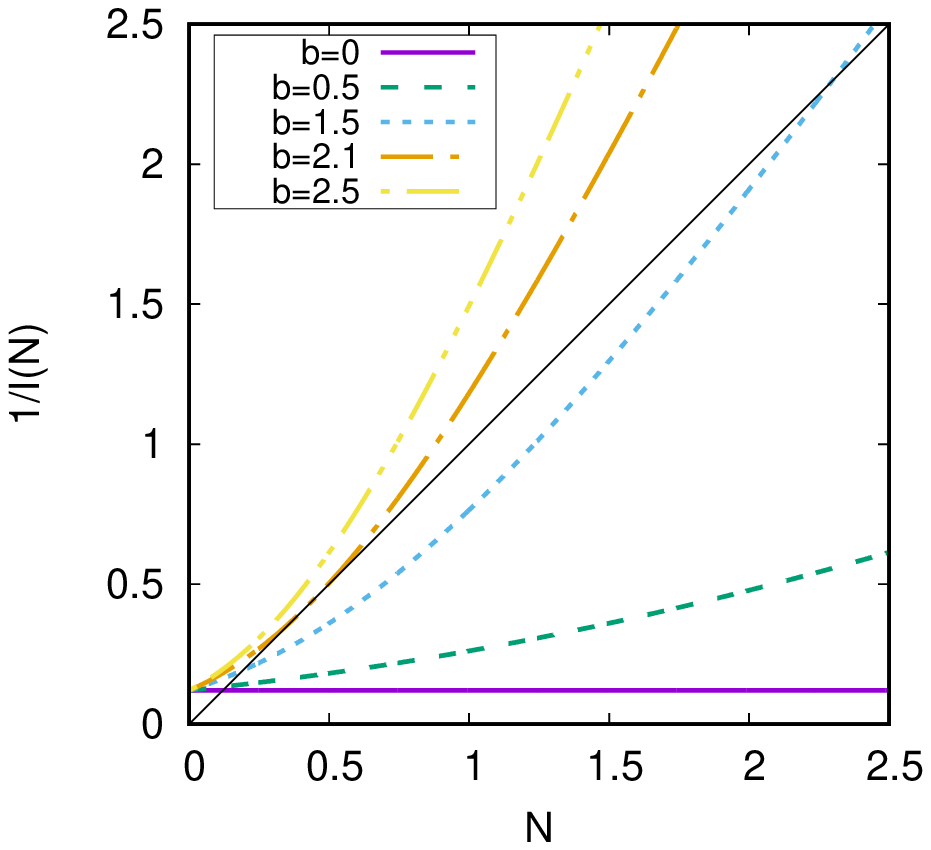}
			\end{center}
		\end{minipage}
	\end{center}
	\caption{{\bf Function $\boldsymbol{\frac{1}{I(N)}}$ (see
            \eqref{eq: I}) for different values of the connectivity
            parameter $\boldsymbol{b}$.}}
	\label{fig: function_I}
\end{figure}
   \begin{figure}[h]
	\begin{center}
		\includegraphics[width=0.49\linewidth]{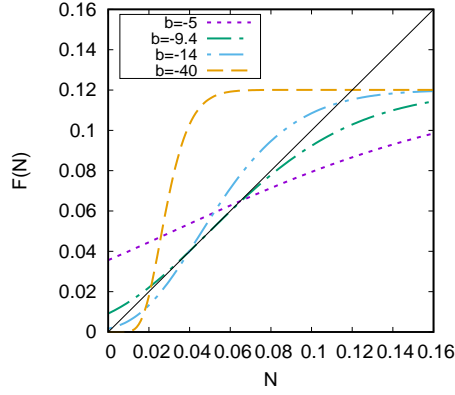}
	\end{center}
	\caption{{\bf Function $\boldsymbol{F:=f\circ f}$ with
            $\boldsymbol{f(N)=1/I(N)}$ (see \eqref{eq: I}) for
            different (negative) values of the connectivity parameter
            $\boldsymbol{b}$.}}
	\label{fig:cycles}
\end{figure}     
We recall that the number of solutions to the Equation \eqref{eq:
  implicit} gives us the number of steady states of the nonlinear
system \eqref{eq:edp-delay}. This number depends on the value of $b$
(see \cite{caceres2011analysis}).  For the excitatory case ($b>0$)
there is only one steady state if $b\le V_F-V_R$. If $V_F-V_R< b$
there are two possibilities: there are at least two steady states
(numerically no more than two steady states have been observed), or,
if $b$ is large enough, there is not steady states. The case
$b=V_F-V_R$ is very interesting because this is the limit value for
which the notion of physical solutions makes sense.  Numerically this
case was analysed in \cite{CR-L} and without delay ($d=0$) the
particle system evolves to a plateau distribution, if the initial
condition is concentrated enough around $V_F$.  In a limit sense, this
plateau profile is a stationary state of the system for
$b=V_F-V_R$. However, the system tends to the unique steady state with
bounded firing rate, if a large transmission delay is considered. For
the inhibitory case ($b<0$) there is only one steady state.
Therefore, Theorems \ref{th:relations-Nb-positive} and
\ref{th:relations-Nb-negative} can also be read in terms of the number
of the steady states of the nonlinear system \eqref{eq:edp-delay}.

\begin{rem}
	Theorem \ref{th:relations-Nb-positive} does not include
	the case with more than two equilibria,
	which can be easily extended following the same strategy. 
	For instance, if Equation \eqref{eq: implicit} has three
	solutions (i.e.  the nonlinear system \eqref{eq:edp-delay} has
	three steady states), $N^*_1<N^*_2<N^*_3$, the sequence
	$\left\{N_{k,\infty}\right\}_{k\ge 0}$ tends to $N^*_1$, if
	$N_{0,\infty} \in (0,N^*_2)$, and to $N^*_3$, if
	$N_{0,\infty} \in (N^*_2,+\infty)$.
	In general, the behaviour of the sequence
	$\left\{N_{k,\infty}\right\}_{k\ge 0}$ depends on whether
	$N_kI(N_k)<1$ or $1<N_kI(N_k)$. In the first case the sequence
	is increasing, while in the second case the sequence is decreasing.
	When the sequence is decreasing it converges to some solution to
	Equation \eqref{eq: implicit}, while if the sequence is increasing
	it converges or diverges depending on whether it is bounded
	by a solution to the implicit equation or not.
\end{rem}

The following theorem shows that the  pseudo-equilibria sequence 
$\left\{p_{k,\infty}(v)\right\}_{k\ge 0}$
described in \eqref{eq:sequence-p} tends to a stationary solution
$p_\infty$ of the nonlinear system \eqref{eq:edp-delay}
if its related sequence $\left\{N_{k,\infty}\right\}_{k\ge 0}$
(see \eqref{eq:sequence-N}) converges to a finite value $N_\infty$,
which is the firing rate of $p_\infty$.
\begin{thm}\label{thm:delta_convergence}
  Consider any nonnegative sequence
  $\left\{N_{k,\infty}\right\}_{k\ge 0}$, and its related
  pseudo-equilibria sequence $\left\{p_{k,\infty}(v)\right\}_{k\ge 0}$
  described in \eqref{eq:sequence-p}. Assume
  $\lim_{k\to\infty} N_{k,\infty}= N_\infty < +\infty$. Then there
  exists $k_0\in \N$ such that for all $k\ge k_0$ and all $N \in \R$
  the following inequalities hold:
  \begin{align}
    \label{eq: relation pk-Nk}
    \| p_{k+1,\infty}(v)-p_{k,\infty}(v) \|_\infty
    &\le
    C_{N_\infty} \vert N_{k,\infty}- N_{k-1,\infty}\vert,
    \\
    \label{eq: relation- d_v pk-Nk}
    \| \p_vp_{k+1,\infty}(v)-\p_vp_{k,\infty}(v) \|_\infty
    &\le
    C_{N_\infty} \vert N_{k,\infty}- N_{k-1,\infty}\vert,
    \\
    \label{eq:pk-Nk-L2}
    \| p_{k+1,\infty}(v)-p_{k,\infty}(v) \|_{L^2(\varphi_N)}
    &\le
    C_{N_\infty} \vert N_{k,\infty}- N_{k-1,\infty}\vert,
    \\
    \label{eq:pk-Nk-H2}
    \| \p_vp_{k+1,\infty}(v)-\p_vp_{k,\infty}(v) \|_{L^2(\varphi_N)}
    &\le
    C_{N_\infty} \vert N_{k,\infty}- N_{k-1,\infty}\vert,
\end{align}
where $C_{N_\infty}>0$ is a constant that depends only on $N_\infty$,
$b$, $V_R$ and $V_F$, and also on $N$ in the case of
(\ref{eq:pk-Nk-L2})--(\ref{eq:pk-Nk-H2}).  Similarly, for all
$k \geq k_0$ and all $N \in \R$,
  \begin{align}
    \label{eq:pk-Nk-inf}
    \| p_{k,\infty}(v)-p_{\infty}(v) \|_\infty
    &\le
    C_{N_\infty} \vert N_{k-1,\infty}- N_{\infty}\vert,
    \\
    \label{eq:d_vpk-Nk-inf}
    \| \p_vp_{k,\infty}(v)-\p_vp_{\infty}(v) \|_\infty
    &\le
    C_{N_\infty} \vert N_{k-1,\infty}- N_{\infty}\vert,
    \\
    \label{eq:pk-Nk-L2-inf}
    \| p_{k,\infty}(v)-p_{\infty}(v) \|_{L^2(\varphi_N)}
    &\le
    C_{N_\infty} \vert N_{k-1,\infty}- N_{\infty}\vert,
    \\
    \label{eq:pk-Nk-H2-inf}
    \| \p_vp_{k,\infty}(v)-\p_vp_{\infty}(v) \|_{L^2(\varphi_N)}
    &\le
    C_{N_\infty} \vert N_{k-1,\infty}- N_{\infty}\vert,
  \end{align}
  with $p_\infty = p_\infty(v)$ the steady state of the nonlinear
  system \eqref{eq:edp-delay} with firing rate $N_\infty$.
\end{thm}

\begin{rem}
  In other words, this result states that pseudo-equilibrium $p_{k,\infty}$
  given in \eqref{eq:sequence-p}, associated to a number $N_{k-1}$,
  depends continuously on $N_{k-1}$ in a wide variety of norms
  including the Sobolev $W^{1,\infty}$ norm and weighted $L^2$ and
  $H^2$ norms. We have chosen to state it for a sequence, since it is
  the exact result which will be later used in Section
  \ref{sec:analytical-results}.
\end{rem}

\begin{proof}
  We recall the expression of the pseudo-equilibrium 
  \begin{equation*}\label{eq: equilibrium-linear}
    p_{k,\infty}(v) = N_{k,\infty}e^{-\frac{\left(v-bN_{k-1,\infty}\right)^2}{2}}
    \int_{\max(v,
      V_R)}^{V_F}e^{\frac{\left(w-bN_{k-1,\infty}\right)^2}{2}}\d w,
  \end{equation*}
  where $N_{k,\infty}=I(N_{k-1,\infty})^{-1}$ and
  $I\left(N\right)=\int_{-\infty}^{V_F}
  e^{-\frac{\left(v-bN\right)^2}{2}}
  \int_{\max(v, V_R)}^{V_F}e^{\frac{\left(w-bN\right)^2}{2}} \d w \d v $.
  We define
  \begin{equation*}
    g(v,N) := e^{-\frac{\left(v-bN\right)^2}{2}} \int_{\max(v,
      V_R)}^{V_F}e^{\frac{\left(w-bN\right)^2}{2}}\d w
    ,
    \qquad
    h(v,N) := \frac{g(v,N)}{I(N)},
  \end{equation*}
  so that $h(v,N_{k-1,\infty}) = p_{k,\infty}(v)$ and
  $h(v,N_\infty) = p_\infty(v)$ (since it must hold
  $1 = N_\infty I(N_\infty)$). Through a first-order Taylor expansion
  in $N$ we can write the difference between two consecutive elements
  of the sequence $ \left\{p_{k,\infty}\right\}_{k\ge 0} $ as
  \begin{equation*}
    p_{k+1,\infty}(v)-p_{k,\infty}(v)=\p_Nh(v,\xi_k)\left(N_{k,\infty}-N_{k-1,\infty}\right),
  \end{equation*}
  where $ \xi_k $ is a value located between $ N_{k-1,\infty} $ and
  $ N_{k,\infty} $. Due to the convergence of the sequence
  $N_{k,\infty}$, there exists $k_0\in \N$ such that $N_{k-1,\infty}$
  and $N_{k,\infty}\in (N_\infty/2, 2 N_\infty)$ for all $k\ge
  k_0$. In particular, $\xi_k \in (N_\infty/2, 2 N_\infty)$ for all
  $k\ge k_0$. Therefore, it is enough to show that
  \begin{equation}
    \label{eq:1}
    |\p_Nh(v,N)| < C_{N_\infty},
  \end{equation}
  uniformly in $v\in (-\infty,V_F]$ and
  $N \in (N_\infty/2, 2 N_\infty)$, with some constant $C_{N_\infty}$
  which depends only on $N_\infty$, $b$, $V_R$ and $V_F$. This would prove inequality
  \eqref{eq: relation pk-Nk}.

  By computing $\p_Nh(v,N) $ we obtain
  \begin{equation}\label{eq:h_prime}
    \p_Nh(v,N)=\frac{\p_Ng(v,N)}{I(N)}-\frac{I'(N)}{I(N)^2}g(v,N).
  \end{equation}
  In order to show (\ref{eq:1}) it is enough to study each term: from
  the expression of $g(v,N)$ one sees that it is bounded as needed,
since $e^{-\frac{(v-bN)^2}{2}}$ and $ve^{-\frac{(v-bN)^2}{2}}$
are  uniformly  bounded in $v\in (-\infty,V_F]$ and
  $N \in (N_\infty/2, 2 N_\infty)$ ;
  $I(N)$ is uniformly bounded below for
  $N \in (N_\infty/2, 2 N_\infty)$, and $|I'(N)|$ (which can be
  explicitly written) is bounded above in the same interval. Finally,
  one can explicitly write $\p_Ng(v,N)$,
  $$
  \p_Ng(v,N)=b(v-bN)g(v,N)-be^{-\frac{(v-bN)^2}{2}}
  \left( e^{\frac{(V_F-bN)^2}{2}}-e^{\frac{(\max(v,V_R)-bN)^2}{2}}\right),
  $$
  and show it is also uniformly
  bounded in the needed range of $v$ and $N$. This shows \eqref{eq:
    relation pk-Nk}. With a completely analogous calculation we prove
  (\ref{eq:pk-Nk-inf}), since
  \begin{equation*}
    p_{k,\infty}(v)-p_{\infty}(v) = \p_Nh(v,\xi_k) \left(N_{k-1,\infty}-N_{\infty}\right),
  \end{equation*}
  for some $\xi_k$ between $N_{k-1,\infty}$ and $N_\infty$. By
  checking that $\|\p_N h(\cdot, M)\|_{L^2(\varphi_{N})}$ is bounded
  uniformly for $M \in (N_\infty/2, 2 N_\infty)$ we also prove
  (\ref{eq:pk-Nk-L2}), and very similarly (\ref{eq:pk-Nk-L2-inf}).

  For the derivatives with respect to $v$ we write
  \begin{equation*}
    \p_v p_{k, \infty} (v)
    = \frac{1}{I(N_{k-1, \infty})} \p_v g(v, N_{k-1.\infty})
    =: \widetilde{h}(v,N_{k-1,\infty})
  \end{equation*}
  and
  \begin{equation*}
    \p_v p_{k+1, \infty} (v)-
    \p_v p_{k, \infty} (v)
    = \p_N \widetilde{h}(v, \xi_k) (N_{k,\infty} - N_{k-1,\infty}).    
  \end{equation*}
  A similar analysis of $\p_N \widetilde{h}(v, N)$ proves the
  remaining points (\ref{eq: relation- d_v pk-Nk}),
  (\ref{eq:pk-Nk-H2}), (\ref{eq:d_vpk-Nk-inf}) and
  (\ref{eq:pk-Nk-H2-inf}).
\end{proof}

In the inhibitory case, the firing rate sequence may converge to a
2-cycle (see Theorem \ref{th:relations-Nb-negative}).  We show the
long-term behaviour of the pseudo-equilibria sequences in those cases
in the following theorem.

\begin{thm}
  \label{thm:two_cicle_convergence}
  Let us consider the firing rate sequence
  $\left\{N_{k,\infty}\right\}_{k\ge 0}$ given in
  \eqref{eq:sequence-N}, and its related pseudo-equilibria sequence
  $\left\{p_{k,\infty}(v)\right\}_{k\ge 0}$ described in
  \eqref{eq:sequence-p}. Assume the pseudo-equilibria sequence
  $ \left\{N_{k,\infty}\right\}_{k\ge0} $ tends to the 2-cycle
  $ \left\{N^-,N^+\right\} $. Then there exists $k_0\in \N$ such that
  for all $k\ge k_0$ the following inequalities hold:
  \begin{equation} \label{eq:relation-1pk-Nk-1}
    \Vert p_{2k,\infty}(v)-p_{2k-2,\infty}(v)\Vert_\infty \le C_{N^-} \vert N_{2k-1,\infty}-
    N_{2k-3,\infty}\vert, 
  \end{equation}
  \begin{equation} \label{eq:relation-pk-Nk-2}
    \Vert p_{2k+1,\infty}(v)-p_{2k-1,\infty}(v)\Vert_\infty \le C_{N^+} \vert N_{2k,\infty}-
    N_{2k-2,\infty}\vert, 
  \end{equation}
  where $C_{N^-},C_{N^+}>0$ depend on $N^-$ and $N^+$, respectively, and $b$, $V_F$ and $V_R$.
	Similarly, for all $k \ge k_0$ 
	\begin{equation*}
              \Vert p_{2k,\infty}(v)-p^-(v)\Vert_\infty \le C_{N^-} \vert N_{2k-1,\infty}-
    N^-\vert, 
	\end{equation*}
	\begin{equation*}
           \Vert p_{2k+1,\infty}(v)-p^+(v)\Vert_\infty \le C_{N^+} \vert N_{2k,\infty}-
    N^+\vert, 
        \end{equation*}
        where $p^-(v)$, $p^+(v)$ are pseudo-equilibria of the
        nonlinear system \eqref{eq:edp-delay} (see
        \eqref{eq:profile-plateau}), with $p^-$ associated to $N^+$
        and $p^+$ associated to $N^-$.	
\end{thm}
\begin{proof}
  This is a direct consequence of Theorem \ref{thm:delta_convergence},
  by considering the sequences $\{N_{2k-1, \infty}\}_{k \geq 1}$ and
  $\{N_{2k, \infty}\}_{k \geq 0}$, which converge to  $N^-$ and $N^+$, respectively.
  The associated pseudo-equilibria
  sequences are then $\{p_{2k, \infty}\}$ and $\{p_{2k+1, \infty}\}$,
  respectively, and the statement is a consequence of Theorem
  \ref{thm:delta_convergence} applied to them.
\end{proof}
We point out that the behaviour of the pseudo-equilibria sequence is
determined by the limit of the firing rate sequence, in case it exists, for
all $b\in \R$. And there is a relation between the nonlinear system
\eqref{eq:edp-delay} and that long-term behaviour:
	\begin{itemize}
		\item In the excitatory case ($b>0$) the  pseudo-equilibria
		sequence $\{p_{k,\infty}(v)\}_{k\ge 0}$:
		\begin{itemize}
			\item converges to the unique steady state of the nonlinear system
			\eqref{eq:edp-delay}, if $b$ is small.
		      \item converges to the steady state
                        with lower firing rate
			of the nonlinear system \eqref{eq:edp-delay}, if the
			system has two stationary solutions and if $\{N_{k,\infty}\}_{k\ge 0}$ has
			finite limit.
                      \end{itemize}
	      \item In the inhibitory case ($b<0$) the pseudo-equilibria
                sequence 
		$\{p_{k,\infty}(v)\}_{k\ge 0}$ tends to the unique stationary solution
		$ p_\infty(v) $ of the nonlinear system    \eqref{eq:edp-delay},
                if $b^*<b$. Otherwise,
		if $b<b^*$, it tends to a 2-cycle
		$ \left\{p^-(v), p^+(v)\right\} $, which are pseudo-equilibria of
		the nonlinear system  \eqref{eq:edp-delay}.
                \end{itemize}

        To prove the convergence of the pseudo-equilibria sequence we use the
        fact that the limit of $\{N_{k,\infty}\}_{k\ge 0}$ is finite or is a
        2-cycle, so it could not be used
	in case the sequence diverges. However, in that case, it could be prove
	that the sequence of pseudo-equilibria $\{p_{k,\infty}(v)\}_{k\ge 0}$
	tends to plateau distribution (point-wise in
        $(-\infty,V_R)\cup (V_R,V_F)$).

\
        
In the following section we use Theorem \ref{thm:delta_convergence} to
prove the convergence to equilibrium of solutions to the nonlinear
system \eqref{eq:edp-delay} in the weakly connected case, by following
the associated pseudo-equilibria sequence. As presented, this
technique only works in weakly connected networks, but it might be
possible to use it for a wider range of $b$.

\section{Convergence to equilibrium along the pseudo equilibria
  sequence for weakly connected networks}
\label{sec:analytical-results}

In this section we study the long-term behaviour of the nonlinear
system \eqref{eq:edp-delay}, con\-si\-de\-ring large transmission
delay values, by following the pseudo-equilibria \eqref{eq:sequence-p}
for weakly connected networks.  To do this, we consider the solution
to the Cauchy problem associated with \eqref{eq:edp-delay} (remember
we assume $a=1$):
\begin{equation}
  \left\{
    \begin{array}{l}
      \p_t p(v, t)+
      \p_v\left[\left(-v+bN(t-d)\right)p(v,t)\right]-\p^2_v p(v,t)
      =\delta (v-V_R)N(t),
      \\
      N(t)=-\p_v p(V_F,t), \ \mbox{for}  \ t\ge 0,
      \\
      p(0,v)=p_0(v),
      \ \mbox{and,} \ N(t)=- \p_v p_0(V_F) \quad t \in [-d,0],
    \end{array}
  \right.
  \label{eq: large-delta}         
\end{equation}
We view this solution as a sequence of functions, considering time
intervals of length $d$. Our initial condition is always a constant on
$[-d,0]$, and we observe that the system becomes linear for
$0\le t<d$, since $N(t-d)$ is constant.  Therefore, for $0\le t<d$,
the Cauchy problem \eqref{eq: large-delta} behaves like a linear
problem of the form
\begin{equation}
  \label{eq: linear}         
  \left\{
    \begin{array}{l}
      \p_t p(v, t)+
      \p_v\left[\left(-v+b\overline{N}\right)p(v,t)\right]-\p^2_vp(v,t)=\delta (v-V_R)N(t),
      \\
      N(t)=- \p_v p(V_F,t), \quad t\ge 0, 
      \\
      p(0,v)=p_0(v), \ \mbox{with} \ \overline{N}=-\p_v p_0(V_F)\ge 0. 
    \end{array}
  \right.
\end{equation}
\begin{figure}[h]
  \begin{center}
		\begin{minipage}[c]{0.50\linewidth}
                  \includegraphics[width=\textwidth]{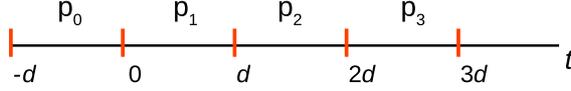}
                  \end{minipage}
  \end{center}
  \caption{{\bf Schematic representation of the solution
      to the Cauchy problem \eqref{eq: large-delta}  in time through
      the sequence $\boldsymbol{p_{k}(v,t)}$ given by \eqref{eq:p}.}}
  \label{fig: scheme}
\end{figure}
During that interval of time, $[0,d)$, we obtain $N(t)$, which appears
in the drift term of the following period of time with size $d$,
$[d,2d)$. Proceeding in the same way for the following time intervals
$[kd,(k+1)d)$ with $k=2,3, \ldots$, the nonlinear problem \eqref{eq:
  large-delta} is equivalent to a sequence of linear problems of the
type
\begin{equation*}
  \left\{
    \begin{array}{l}
      \p_t p(v, t)+
      \p_v\left[h\left(v,\overline{N}(t)\right)p(v,t)\right]-\p^2_vp(v,t)=\delta (v-V_R)N(t),
      \\
      N(t)=-\p_v p(V_F,t), \quad t\ge 0, 
    \end{array}
  \right.
  \label{eq: linear-general}         
\end{equation*}
where $\overline{N}(t)$ is a known function.  To better describe the
idea consider the following notation (see Figure \ref{fig: scheme}):
\begin{equation}\label{eq:p}
  p_{k+1}(v,t):= p(v,t+kd),
  \ \text{with} \ t\in\left[0,d\right], \
  v\in(-\infty,V_F]  \
  \text{and} \ k=0,1,2...
\end{equation}
In other words, $p_{k+1}(v,t)\equiv p(v,\bar{t})$, with
$\bar{t}:=t+kd\in\left[kd,(k+1)d\right)$. Similarly,
\begin{equation*}
  N_k(t) := N(t + kd),
  \qquad t \in [-d,0], k = 0, 1, 2, \dots
\end{equation*}
Hence we write the nonlinear problem \eqref{eq: large-delta} in any
interval $ (kd,(k+1)d) $, as follows: for $t\in
(0,d)$ 
and $k=1,2,\ldots$
\begin{equation}
  \left\{
    \begin{array}{l}
      \p_t p_k(v, t)+
      \p_v\left[\left(-v+bN_{k-1}(t-d)\right)p_k(v,t)\right]-\p^2_v p_k(v,t)=\delta (v-V_R)N_k(t),
      \\
      N_k(t)=-\p_v p_k(V_F,t), \ p_k(V_F)=0, 
    \end{array}
  \right.
  \label{eq:system-k}
\end{equation}
Its related stationary problem is given by
\begin{equation}
  \left\{
    \begin{array}{l}
      \p_v\left[\left(-v+bN_{k-1,\infty}\right)p_{k,\infty}(v)\right]-\p_v^2 p_{k,\infty}(v)=\delta (v-V_R)N_{k,\infty},
      \\
      N_{k,\infty}=- \p_v p_{k,\infty}(V_F), \ \text{and} \
      \int_{-\infty}^{V_F}p_{k,\infty}(v) \d v=1,
    \end{array}
  \right.
  \label{eq: system-k-stationary} 
\end{equation}
whose unique solution is the pseudo-equilibria sequence (see
\eqref{eq:sequence-p})
\begin{equation}\label{eq: solution-system-k-stationary}
  p_{k,\infty}(v) = N_{k,\infty}e^{-\frac{\left(v-bN_{k-1,\infty} \right)^2}{2}}
  \int_{\max(v, V_R)}^{V_F}e^{\frac{\left(w-bN_{k-1,\infty}
      \right)^2}{2}} \d w,
\end{equation}
given in terms of the firing rate sequence \eqref{eq:sequence-N}, with
$N_{0,\infty}:=-\p_vp_0(V_F)$.

Our purpose is to show that one may prove convergence to equilibrium
of bounded solutions $p$ to the nonlinear system by following the
sequence of pseudo-equilibria $p_{k,\infty}$. For this we need to
assume several properties of the linear system, which are reasonable
in view of our recent results in \cite{local}. In order to describe
these assumptions we consider the space $X$ given by
\begin{equation}
  \label{eq:space-X}
  X := \{ u \in \mathcal{C}(-\infty, V_F] \cap \mathcal{C}^1(-\infty, V_R]
  \cap \mathcal{C}^1[V_R, V_F] \mid u(V_F) = 0 \ \text{and} \ \|u\|_X < \infty\},
\end{equation}
with
\begin{equation}\label{eq:norm-X}
  \|u \|_X :=
  \|u \|_\infty + \|\p_vu \|_\infty
  + \|u \|_{L^2(\varphi)} + \|\p_vu \|_{L^2(\varphi)}
\end{equation}
and
\begin{equation*}
  \varphi(v)
  := \exp\left( \frac{v^2}{2} \right),
  \qquad v \in (-\infty, V_F].
\end{equation*}
This is the space of continuous functions on $(-\infty, V_F]$ which
are ${\cal C}^1$ except possibly at $v=V_R$, where they must still
have one-sided derivatives; and which are in the space $H^1$ with the
Gaussian weight $\varphi$ (so they are strongly decaying functions for
$v \to -\infty$). The main merit of this space is that the firing rate
$-\p_v p(t, V_F)$ is a continuous operator in this norm.
Our main assumptions on the associated linear equation are the following:
\begin{enumerate}
\item \textbf{(Spectral gap.)} The semigroup $e^{tL}$ associated to
  the linear equation \eqref{eq: linear} has a spectral gap in the
  space $X$. That is, there exist constants $\lambda > 0$, $C \geq 1$
  such that for all initial conditions $u_0 \in X$ with zero integral,
  it holds that
  \begin{equation}
    \label{eq:gap-X}
    \| e^{t L} u_0 \|_{X}
    \leq Ce^{-\lambda t} \| u_0 \|_{X}
    \qquad \text{for all $t \geq 0$.}
  \end{equation}
  We assume this property holds uniformly for any $\bar{N}$ close to an
  equilibrium firing rate $N_\infty$ of the nonlinear equation \eqref{eq:edp-delay}.
\item \textbf{(Regularization property.)}  There exist constants
  $\lambda>0$, $\tilde{C} \geq 1$ such that for all initial conditions
  $u_0\in L^2(\varphi)$ with zero integral, the following inequality
  holds:
  \begin{equation}\label{eq:gap-reg-X}
    \Vert e^{tL}u_0\Vert_X
    \le \tilde{C} t^{-3/4}e^{-\lambda t}
    \Vert u_0 \Vert_{L^2{(\varphi)}}
    \qquad \text{for all $ t>0 $}.
  \end{equation}
  Again, we assume this property holds uniformly for any $\bar{N}$ close to an
  equilibrium firing rate $N_\infty$ of the nonlinear equation \eqref{eq:edp-delay}.
  We point out that the exponent $-3/4$ here is the same as for the
  standard Fokker-Planck equation $$\p_t p = \p_v^2 p + \p_v(v p).$$
\end{enumerate}
Similar results to these assumptions are given in
\cite{local}.
A proof of them can be given by using
the techniques  developed there, but this is not the aim of this paper
and we defer a more detailed study of these spectral properties to a
future work.

\medskip
We will prove the following result:

\begin{thm}\label{thm:convergence}
  Take $b\in\R$. Let us consider an initial condition $p_0 \in X$ for
  the nonlinear system \eqref{eq:edp-delay} such that the firing rate
  sequence $\left\{N_{k,\infty}\right\}_{k\ge0}$ with initial
  condition $N_{0,\infty}=-\p_v p_0(V_F)$ converges to a certain value
  $N_\infty > 0$ (which must then satisfy $N_\infty I(N_\infty) =
  1$). Let $p_\infty$ be the stationary solution to
  \eqref{eq:edp-delay} with firing rate $N_\infty$. Assume the
  spectral gap and regularisation properties stated before this
  theorem. Let $p = p(v,t)$ be the solution to \eqref{eq:edp-delay}
  with initial data $p(v,t) = p_0(v)$ for all $v \in (-\infty, V_F]$
  and all $t \in [-d,0]$. Let us assume that there exists $K > 0$ such
  that
  \begin{equation}
    \label{eq:p-bounded}
    \|p(.,t)\|_X\le K
    \qquad \text{for all $t\ge 0$.}
  \end{equation}
  Then there exist $d_0, b_0, Q, \mu>0$ such that
  the solution $p$ to the nonlinear system \eqref{eq:edp-delay} with
  $d>d_0$, $|b|<b_0$, and initial condition $p_0$ satisfies
  \begin{equation}
    \|p(.,t)-p_\infty(.)\|_X\le Qe^{-\mu t}\|p_0-p_\infty\|_X
    \qquad \text{for all $t\ge 0$.}
  \end{equation}
\end{thm}

\begin{rem}
  Assumption (\ref{eq:p-bounded}) in Theorem \ref{thm:convergence}
  merits an explanation. Our result only gives the behaviour of
  solutions which are \emph{uniformly bounded in time} (which is
  consistent with convergence to equilibrium). It is known that
  solutions may blow up if the delay $d=0$, and when $d > 0$ we do not
  know whether there may be solutions with diverging values of $N(t)$
  as $t \to +\infty$. Our result applies only to solutions whose
  firing rate $N(t)$ is uniformly bounded for all times.  Once we know
  the firing rate $N(t)$ is bounded it may be possible to carry out
  regularisation estimates to show that $\|p(t,\cdot)\|_X$ is
  uniformly bounded for all times, but we assume the latter stronger
  condition to avoid these technical details.
\end{rem}

The proof of Theorem \ref{thm:convergence} is based on: 1) our
spectral hypotheses on $e^{tL}$, 2) the uniform boundedness hypothesis
on $\|p(t,\cdot)\|_X$ and 3) Theorem \ref{thm:delta_convergence},
which shows that
\begin{equation} \label{eq:delta-k}
  \| p_{k+1,\infty}-p_{k,\infty}\|_X \le C_{N_\infty} \vert N_{k,\infty}-
  N_{k-1,\infty}\vert,
\end{equation}
and, therefore $ \| p_{k+1,\infty}-p_{k,\infty}\|_X \to 0$ as
$\lim_{k\to\infty} N_{k,\infty}= N_\infty$. We will also need the
following two elementary lemmas. The first one is a discrete version
of the variation of constants technique, which can be checked in a
straightforward way and we give without proof:
\begin{lem}[Discrete variation of constants]
  \label{lem:variation-of-constants}
  Let $M$ be a linear operator on a certain vector space $E$, and
  $\{b_k\}_{k\ge 0}$ any sequence in $E$. Then, given $a_0 \in E$, the
  sequence
  \begin{equation*}
    a_k := M^ka_0+\sum_{i=0}^{k-1}M^{k-i-1}b_i, \quad k=1,2,\ldots
  \end{equation*}
  is the (unique) solution to the linear equation
  \begin{equation}\label{eq:discrete-variation-1}
    a_k=Ma_{k-1}+b_{k-1}, \quad k=1,2,\ldots.
  \end{equation}
\end{lem}

The second lemma needed for the proof of Theorem \ref{thm:convergence}
is a discrete version of Gronwall's Lemma.

\begin{lem}[Discrete Gronwall's Lemma]\label{lem:gronwall} If
  $ \phi_k $ is a positive sequence and $W \in \R$, $V\geq 0$ are
  constants such that
 $   \phi_k \leq W+V\sum_{i=0}^{k-1}\phi_i$ for $k\ge0$,
  then
  $ \phi_k \leq We^{kV} $
  for all $k \geq 0$.
\end{lem}
\begin{proof}
  To prove the result we define the sequence
  $ \psi_k:=W+V\sum_{i=0}^{k-1}\phi_i $ for $k\ge 1$ and $ \psi_0:=W$.
  We compute de difference between two consecutive elements of that sequence and then we write it down in terms  of  the sequence
  $\phi_k$:
  \begin{equation*}
    \psi_k-\psi_{k-1} = V\phi_{k-1} \le V\left(W+V\sum_{i=0}^{k-2}\phi_i\right)=V\psi_{k-1}, 
  \end{equation*}
  which leads to
  $$
  \phi_{k-1} \le \psi_{k-1}, 
  $$
  and
  \begin{equation*}
    \psi_k \le (1+V)\psi_{k-1} \implies \ \psi_k \le\left(1+V\right)^k\psi_0=\left(1+V\right)^kW.
  \end{equation*}
  Then we have the result, since $ \left(1+V\right)\le e^V $:
  \begin{equation*}
    \phi_k\le\left(1+V\right)^kW<We^{kV}.
    \qedhere
  \end{equation*}
\end{proof}

\begin{proof}[Proof of Theorem \ref{thm:convergence}]
  The proof is based on the study of the nonlinear system \eqref{eq:edp-delay}
  in time intervals of size $d$, in which the system becomes linear. To do that
  we start with the nonlinear system \eqref{eq:system-k}
  for $t\in (0,d)$ 
  and $k=1,2,\ldots$,
  and
  define
  \begin{equation*}
    u_k(v,t) := p_k(v,t) - p_{k,\infty}(v),
  \end{equation*}
  the differences to the pseudo-equilibria $p_{k,\infty}(v)$. With
  this notation we rewrite the nonlinear system \eqref{eq:system-k},
  splitting the equation in a linear part plus a nonlinear part:
  \begin{equation*}
    \p_t u_k(v,t) = L_{k-1} u_k(v,t) + b (N_{k-1, \infty} - N_{k-1}) \p_v p_{k}(v,t),
  \end{equation*}
  with the same Dirichlet boundary condition as before and defining
  the linear operator $L_{k-1}$, associated to the firing rate
  $N_{k-1, \infty}$, acting on $ u=u(v) $, by
  $$L_{k-1}u:=\p_v (v-bN_{k-1,\infty}u)+\p_v^2u+\delta (v-V_R)N_u,$$
  where $N_u:=-\p_vu(V_F)$ emphasises that $ N_u $ is the firing rate
  associated to $ u $. By Duhamel's formula we get
  \begin{equation*}
    u_k(v,t) = e^{t L_{k-1}} u_{k}(v,0)
    +
    b \int_0^t (N_{k-1, \infty} - N_{k-1}(s)) e^{(t-s) L_{k-1}} \p_v p_{k}(v,s) \d s.
  \end{equation*}
  Taking the $X$ norm we note that
  $|N_{k-1, \infty} - N_{k-1}(t)|\leq \| u_{k}(.,t) \|_X$, and using the spectral
  gap of $L_{k-1}$ in $X$:
  \begin{equation*}
    \| e^{t L_{k-1}} u_0  \|_{X}
    \leq C e^{-\lambda t} \| u_0  \|_{X}, \quad 1\le C,
  \end{equation*}
  we have, for all $ t\ge0 $:
  \begin{multline*}
    \|u_k(.,t)\|_X
    \leq
    C e^{-\lambda t} \| u_{k}(.,0) \|_X
    +
    |b| \int_0^t |N_{k-1, \infty} - N_{k-1}(s)|\, \| e^{(t-s) L_{k-1}} \p_v
    p_{k}(.,s) \|_X \d s
    \\
    \leq
    C e^{-\lambda t} \| u_{k}(.,0) \|_X
    +
    |b| \int_0^t \|u_{k-1}(.,s)\|_X
    \, \| e^{(t-s) L_{k-1}} \p_v p_{k}(.,s) \|_X \d s \\
    \leq
    C e^{-\lambda t} \| u_{k}(.,0) \|_X
    +
    \tilde{C}|b|\int_0^t e^{-\lambda(t-s)} (t-s)^{-3/4}\|u_{k-1}(.,s)\|_X
    \| \p_v p_k(.,s) \|_{L^2(\varphi)}d s,
  \end{multline*}
  where in the last inequality we used third hypothesis of the theorem, written in this particular case as
  $$\Vert e^{(t-s)L_{k-1}}\p_v
  p_{k}(.,s)\Vert_X\le \tilde{C}
  (t-s)^{-3/4}e^{-\lambda (t-s)}\| \p_v p_k(.,s) \|_{L^2(\varphi)}.
  $$
  \begin{rem}
    We are considering the same value of $ \lambda $ for all the
    spectral gaps of the operators $ L_k $ with $ k=1,2,... $, because
    these values come from the Poincare's like inequality used to
    prove the spectral gap of the linear equation in the space
    $L^2_{(p_\infty^{-1})}$ (see
    \cite[Appendix]{caceres2011analysis}).  These values depend only
    on the tails of the pseudo-equilibria $p_{k,\infty}$ and,
    considering that they convergence to $ p_\infty $ (see Theorem
    \ref{thm:delta_convergence})), we may take a value $\lambda$ valid
    for all $k$.
  \end{rem}
  After this we can bound the $L^2(\varphi)$ norm  of the derivative
  of $ p_k $ as
  $$
  \| \p_v p_k(.,t) \|_{L^2(\varphi)} \leq \| \p_v u_k(.,t) \|_{L^2(\varphi)} +
  \| \p_v p_{k,\infty}(.) \|_{L^2(\varphi)}
  \leq \| u_k(.,t) \|_X  + \bar{C},
  $$
  with $ \bar{C}>0 $, and, using \eqref{eq:p-bounded},
  and denoting the new constant again by $K$, we have
  $ \|u_{k}(.,s)\|_X\le K<\infty \;\forall s \in
  \left[0,t\right]\subseteq [0,d]$.
  Thus, with the constant out of the integral
  and renaming it as $ C_b:=\tilde{C}\vert b\vert\left(K + \bar{C}\right) $,
  we get
  \begin{equation}\label{eq:inequality-norm-1}
    \|u_k(.,t)\|_X
    \leq
    C e^{-\lambda t} \| u_{k}(.,0) \|_X
    \\+
    C_b \int_0^t e^{-\lambda(t-s)}(t-s)^{-3/4} \|u_{k-1}(.,s)\|_X \d s.
  \end{equation}
  Taking into account that
  $ p_{k}(v,0)=p_{k-1}(v,d) $,
 $ u_{k-1}(v,d):= p_{k-1}(v,d)-p_{k-1,\infty}(v)$,
  and
 $ u_k(v,0):=p_k(v,0)-p_{k,\infty}(v) $,
  we can write $u_k(v,0)=u_{k-1}(v,d)+\left(p_{k-1,\infty}(v)-p_{k,\infty}(v)\right) $,
  which, taking the $ X $ norm in $ v $, leads to
  $$
  \|u_k(.,0)\|_X\leq\|u_{k-1}(.,d)\|_X+\delta_k
  $$
  with $ \delta_k:=\|p_{k-1,\infty}-p_{k,\infty}\|_X $.
  By using the definitions $ f_k(t):=e^{\lambda t}\|u_k(.,t)\|_X $ and $ \epsilon(t):=e^{-\lambda t} $, we rewrite
  the previous inequality as $ f_k(0)\leq e^{-\lambda d}f_{k-1}(d) + \delta_k$, and
  \eqref{eq:inequality-norm-1} as
  \begin{equation}\label{eq:decay-inequality-1pre}
    f_k(t) \leq C \epsilon(d)f_{k-1}(d)+
    C \delta_k +
    C_b\int_{0}^{t}f_{k-1}(s)(t-s)^{-3/4}\d s,
  \end{equation}
  which can be rewritten, in terms of the linear operators $$ Af_k(t):=C f_k(d),\;  Bf_k:=C_b\int_{0}^{t}f_{k}(s)(t-s)^{-3/4}\d s 
  \quad\text{and}\quad   h f_{k-1}:=\epsilon(d) Af_{k-1}, $$ in the following way:
  \begin{equation}\label{eq:decay-inequality-1}
    f_k\leq \left(\epsilon(d) A+B\right)f_{k-1}+C \delta_k=
    h_{k-1}+Bf_{k-1}+C \delta_k,
  \end{equation}
  To prove 
  the decay of $\|u_k(.,t)\|_X$ we shall proceed in two steps:
  first we study the solution to $ f_k\leq   h_{k-1}+Bf_{k-1}$  and prove 
  its convergence.   Secondly we extend the converge to the complete
  sequence $ f_k $, using that $ \delta_k\rightarrow0 $ (see
  Theorem \ref{thm:delta_convergence} and previous comments before the proof).
  
  \
  
  \noindent
  {\it First step: Study of the recurrence $ f_k\leq   h_{k-1}+Bf_{k-1}$.}
  \newline
  \noindent
  We note that if $f_k$ satisfies the inequality
  $f_k\leq   h_{k-1}+Bf_{k-1}$, then $f_k\le x_k$, where $x_k$ is the solution
  to the recursive equation $ x_k= h_{k-1}+Bx_{k-1}$ with initial
  condition $x_k=f_0$. Then,  using the Lemma
  \ref{lem:variation-of-constants} to $x_k$ we obtain
  \begin{equation*}
    f_k\le B^kf_0+\sum_{i=0}^{k-1}B^{k-i-1}h_i.
  \end{equation*}
  Therefore we need to estimate $B^kf_0$ for $k=1,2,\ldots$. After some computations we obtain
  $$
  B^k(f_0)\le \Vert f_0 \Vert_\infty t^{\frac{k}{4}}\frac{\left(C_b\Gamma \left(\frac{1}{4}\right)\right)^k}{\Gamma(1+\frac{k}{4})},
  \quad 0<t<d,
  $$
  by using
  $$
  \int_0^t s^n(t-s)^{-3/4} \d s=t^{n+\frac{1}{4}}\beta (n+1,\frac{1}{4}), \quad n=1,2,\ldots
  $$
  and properties of the Gamma, $\Gamma$, and Beta, $\beta$, functions,
  as
  $$
  B^k(f_0)\le (C_b)^k\Vert f_0 \Vert_\infty t^{\frac{k}{4}}\prod_{i=0}^{k-1}\beta\left(1+\frac{k}{4},\frac{1}{4}\right)=
  (C_b)^k\Vert f_0 \Vert_\infty t^{\frac{k}{4}}\prod_{i=0}^{k-1}\frac{\Gamma(1+\frac{k}{4})\Gamma(\frac{1}{4})}
  {\Gamma(1+\frac{k+1}{4})}.
  $$
Then we use the $\Vert . \Vert_\infty$ norm in $ t\in[0,d] $ so that the following inequality holds:
  \begin{equation*}
    \|f_k\|_\infty\leq\frac{\left(d^{1/4}C_b\Gamma \left(\frac{1}{4}\right)\right)^k}{\Gamma(1+\frac{k}{4})}\|f_0\|_\infty+\sum_{i=0}^{k-1}\frac{\left(d^{1/4}C_b\Gamma \left(\frac{1}{4}\right)\right)^{k-i-1}}{\Gamma\left(1+\frac{k-i-1}{4}\right)}\|h_i\|_\infty,
  \end{equation*}
or, equivalently
    \begin{equation}\label{eq:f_bound_old}
	\|f_k\|_\infty\leq\frac{\left(d^{1/4}C_b\Gamma \left(\frac{1}{4}\right)\right)^k}{\Gamma(1+\frac{k}{4})}\|f_0\|_\infty+C\epsilon(d)\sum_{i=0}^{k-1}\frac{\left(d^{1/4}C_b\Gamma \left(\frac{1}{4}\right)\right)^{k-i-1}}{\Gamma\left(1+\frac{k-i-1}{4}\right)}\|f_i\|_\infty.
	\end{equation}
	Bearing in mind that $ C>1 $, we consider
        $
        C\frac{\left(d^{1/4}C_b\Gamma(1/4)\right)^k}{\Gamma(1+\frac{k}{4})}\le
        \eta_{b,d}e^{-k} $ for an appropriate $ \eta_{b,d} $, which
        can be computed by finding the maximum of the function
	$$ g(k):=\frac{\left(d^{1/4}C_be\Gamma(1/4)\right)^k}{\Gamma(1+\frac{k}{4})}=\frac{M^k}{\Gamma(1+\frac{k}{4})} \quad \text{with} \quad M:=d^{1/4}C_be\Gamma(1/4).$$ 
	We take the logarithm of $ g(k) $ and then we find a quantity that bounds the maximum of function $ g(k) $ by approximating the gamma function using the Stirling's formula
	$$ \Gamma\left(1+\frac{k}{4}\right)\ge\left(\frac{k}{4}\right)^{\frac{k}{4}}e^{-\frac{k}{4}}, $$
	such that $$ \log g(k)\le k\log M + \frac{k}{4}-\frac{k}{4}\log \frac{k}{4}=:\hat{g}(k).$$ Studying the first derivative of function $ \hat{g}(k) $ we compute the maximum, given by $ \hat{g}(4M^4) $ and then bounding the function $ g(k) $ as $g(k) \le  e^{M^4}$. This procedure allows us to define the quantity $ \eta_{b,d} $ as
	\begin{equation*}
	\eta_{b,d}:=Ce^{M^4}=Ce^{d\left(|b|\tilde{C}(\bar{C}+K)e\Gamma(1/4)\right)^4}=Ce^{d|b|^4\hat{C}},
	\end{equation*}
	having unified all constants in the exponential as $ \hat{C}:=\tilde{C}(\bar{C}+K)e\Gamma(1/4) $.
	
	\
	
	Now we can rewrite expression \eqref{eq:f_bound_old} through the following inequality:
	\begin{equation*}
		\|f_k\|_\infty\leq \eta_{b,d}e^{-k}\|f_0\|_\infty+ \eta_{b,d}\epsilon(d)\sum_{i=0}^{k-1}e^{-(k-i-1)}\|f_i\|_\infty.
	\end{equation*}
  Finally, if we define $ \phi_k:=e^k\|f_k\|_\infty $, $ W:=\eta_{b,d}\|f_0\|_\infty $ and $ V:=\eta_{b,d}\epsilon(d) e $, we can write the previous inequality as
  \begin{equation}\label{eq:inequality-before-gronwall}
    \phi_k \leq W+V\sum_{i=0}^{k-1}\phi_i.
  \end{equation}
  Now we use the discrete Gronwall's Lemma \ref{lem:gronwall} to turn the equation \eqref{eq:inequality-before-gronwall} into the following
  $\phi_k \leq We^{kV}$,  which leads to
  \begin{equation}\label{eq:inequality-after-gronwall}
    \|f_k\|_\infty \leq We^{k\left(V-1\right)}.
  \end{equation}
  Equation \eqref{eq:inequality-after-gronwall} implies that $ V $ must be less than $ 1 $ in order to obtain convergence to $ 0 $ of the sequence $ \|u_k(t)\|_X $, so that the condition over the delay is the following:
  \begin{equation*}
    d>\frac{1+log\left(\eta_{b,d}\right)}{\lambda}=\frac{1 + \log C +\hat{C}db^4 }{\lambda},
  \end{equation*}
or, equivalently
\begin{equation}\label{eq:requirement-delay}
	d > \frac{1+\log C}{\lambda - \hat{C}b^4}
\end{equation}
This inequality requires a smallness condition on $b$, since
$\lambda - \hat{C}b^4$ needs to be positive. That is: we can take $d$
satisfying \eqref{eq:requirement-delay}, only if
$|b|^4<\lambda/\hat{C}$.
  
  \
  
  \noindent
  {\it Second step: Study of  $f_k\leq h_{k-1}+Bf_{k-1}+C \delta_k$.}
  \newline
  \noindent
  We write inequality \eqref{eq:decay-inequality-1}  in terms of linear
  operator $\mathcal{M}:=\left(\epsilon(d) A +B\right) $
  as
  \begin{equation*}
    f_k\le \mathcal{M}f_{k-1}+C\delta_k.
  \end{equation*}
  Using the Lemma \ref{lem:variation-of-constants} as before, we get to
  \begin{equation}\label{eq:before_gronwall_2}
    f_k\le \mathcal{M}^kf_0+\sum_{i=0}^{k}\mathcal{M}^{k-i}C\delta_i.
  \end{equation}
  We already know from equation \eqref{eq:inequality-after-gronwall} that
  $ \|\mathcal{M}^kf_0\|_\infty\le C\|f_0\|_\infty e^{k(V-1)} $, so that we can take norm infinity in equation \eqref{eq:before_gronwall_2} and then write it as
  \begin{equation}
    \|f_k\|_\infty\le Ce^{k(V-1)}\|f_0\|_\infty +C^2\sum_{i=0}^{k}e^{(k-i)(V-1)}\|\delta_i\|_\infty.
  \end{equation}
  Then, the condition to obtain convergence to $ 0 $ of $ f_k $ as
  $ k\to\infty $ is given by two different requirements. First, as
  before, condition \eqref{eq:requirement-delay} should be
  satisfied. Secondly, $ \|\delta_k\|_\infty $ has to converge to
  $0$ as $ k\to\infty $, as proven in Theorem
  \ref{thm:delta_convergence}.
\end{proof}

\begin{rem}
  Convergence to the equilibrium was also proven using the entropy
  method in \cite[Theorem 5.3]{caceres2019global} in weakly connected
  networks. In particular, the required smallness of the connectivity
  parameter was $8b^2e^{\lambda d}\le\epsilon/N_\infty$, for
  $\epsilon$ a small constant, depending on $N_\infty$ and constants
  of Sobolev injection of $H^1(I)$ in $L^\infty(I)$ ($I$ is a small
  neighbourhood of $V_R$).  Our bound is $ b^4<\lambda/\hat{C}
  $. With this strategy we obtain the convergence for any large $d$ if
  the smallness of $b$ is satisfied. This could also be compared to a
  similar result in \cite[Theorem 4.1]{local}.
\end{rem}

\section{Numerical results: Global perspective on the long-time
  behaviour of the delayed NNLIF model}
\label{sec:numerical}

In this section we illustrate numerically the relationship between the
discrete pseudo equilibrium model \eqref{eq:sequence-p} and the highly
delayed NNLIF model \eqref{eq:edp-delay}. We give numerical evidence
that the long-time behaviours of the two models are closely
related. In particular, in the long run, we can predict the behaviour
of the nonlinear system by knowing the behaviour of the discrete
system, which was studied in Section
\ref{sec:pseudo-equilibria-sequence}.

The numerical approximation of equation \eqref{eq:edp-delay} has been
carried out by means of a fifth order finite difference flux-splitting
WENO scheme \cite{cockburn1998essentially} for the advection term, a
standard second order finite differences for the diffusion term, and
an explicit third order TVD Runge-Kutta method for the time
evolution. This scheme has been used previously to simulate NNLIF
models \cite{caceres2018analysis} and a detailed explanation of the
scheme can be found in \cite{caceres2011numerical}. Other numerical
schemes, such as those based on a Scharfetter-Gummel reformulation,
have also been used to carry out numerical simulations of this model
\cite{hu2021structure, ikeda2022theoretical}.

\

Our discretisation is composed of a mesh in voltage with
$ v_i=v_{min} + i\Delta v, i=0,1,2,..., n_v$ with a suitable minimum
value $v_{min}$, which ensures the mass is approximately $0$ to the
left of $v_{min}$; and a threshold value $V_F$ as the maximum mesh
value $v_{n_v}$. The mesh in time is given by
$ t_j=j\Delta t,\; j=0,1,2,....,n_t$ where the value of $ \Delta t $
is chosen so that it complies with the CFL condition imposed for a
correct approximation of the drift and diffusion terms
\begin{equation*}
	\Delta t < \min\left(\frac{a(\Delta v)^2}{2},
	\frac{C_{CFL}\Delta v}{\max|bN(t-d)-v|}\right).
\end{equation*}
The reset value $ V_R $ is
one of the nodes of the mesh in voltage, 
and 
the delta function 
in the right term of \eqref{eq:edp-delay}
is approximated by a very
sharp Maxwellian centered in $ V_R $, of the form:
\begin{equation}\label{eq:maxwellian}
	m(x)=\frac{1}{\sqrt{2\pi}\sigma}e^{-\frac{(x-V_R)^2}{2\sigma^2}},
\end{equation}
with $ \sigma = 10^{-6} $, which is normalized by integrating it in our mesh and setting the integral to $ 1 $. The boundary conditions are impose at every time step by setting $ p(v_0) = 0 $ and $ p(v_{n_v})=0 $. Furthermore, we ensure that the values of the probability distribution near $v_0$ are numerically 0, so that there are no issues arising from forcing the boundary condition with the size of the mesh.

The initial conditions we have considered for
these simulations are approximations of:
\begin{itemize}
\item The pseudo-equilibria profiles (see \eqref{eq:sequence-p})
  
  \begin{equation}\label{eq:profiles}
    p(v)=\overline{N}e^{-\frac{(v,bN)^2}{2}}\int_{\max(v,V_R)}^{V_F}e^{\frac{(w-bN)^2}{2}}\d
    w,
  \end{equation}
  fulfilling the condition $ \int_{-\infty}^{V_F}p(v) \d v = 1 $,
  with the particular cases
  \begin{equation}\label{eq:profile-}
		p^-(v) = N^-e^{-\frac{\left(v-bN^+ \right)^2}{2}}
		\int_{\max(v, V_R)}^{V_F}e^{\frac{\left(w-bN^+
                    \right)^2}{2}}\d w,
		\quad N^-=\frac{1}{I(N^+)}
	\end{equation}
	and
	\begin{equation}\label{eq:profile+}
		p^+(v) = N^+e^{-\frac{\left(v-bN^- \right)^2}{2}}
		\int_{\max(v, V_R)}^{V_F}e^{\frac{\left(w-bN^-
                    \right)^2}{2}}\d w,
		\quad N^+=\frac{1}{I(N^-)}, 
	\end{equation}
  the 2-cycle of
  pseudo-equilibria sequence, given by the 2-cycle $\{N^-,N^+\}$ of
  the firing rate sequence.
\item Double Maxwellians:
  \begin{equation}
    \frac{1}{\sqrt{8\pi}\sigma}(e^{\frac{-\left(v-\mu\right)^2}{2\sigma^2}} + e^{\frac{-\left(v + \mu + 2\right)^2}{2\sigma^2}})
    \quad \mu\in \R,  \sigma>0
    .
    \label{eq:2maxw}
  \end{equation}
\end{itemize}
Three system parameters are fixed: $a=1$, $V_R = 1$ and $V_F = 2$.
The connectivity parameter $b$ and the delay $d$ will change depending
on the simulation, displaying different phenomena for this model.

\

We analyse in detail three different aspects: bi-stability between
the lower equilibrium and the plateau distribution for excitatory
networks with two equilibria;  the emergence of periodic solutions
for highly inhibitory systems; and  the influence of the
delay value on the evolution in time of the nonlinear system, in
relation with the behaviour of the pseudo-equilibria sequence.
The rest of the system behaviour is well represented by the above study, as we shall explain below.

\subsubsection*{The case with two equilibria: Bi-stability between the
  lower equilibrium and the plateau distribution}

We consider in this case the regime in which the implicit equation
$NI(N)=1$ has two solutions ($ N_1^*<N_2^* $), which determine the
stationary firing rates of the nonlinear system
\eqref{eq:edp-delay}. This regime corresponds to values of the
connectivity parameter $b$ between $V_F-V_R=1 $ and approximately
$2.3$. 
Specifically, we take $ b=1.5 $ for which $ N_1^*\approx0.194 $ and
$N_2^*\approx2.294$, but we emphasise that other values of $b$ in that
range appear to show an equivalent behaviour. For this connectivity
value, Figure \ref{fig:b_1.5_sequence}
shows the firing rate sequence $N_{k,\infty}$ (see \eqref{eq:sequence-N})
for different initial conditions $N_{0,\infty}$. The behaviour of the sequence was analysed 
in Theorem \ref{th:relations-Nb-positive}. Thus, we know that in that
case (with two equilibria) the firing rate sequence tends to $N_1^*$
or diverges, as we can see in Figure \ref{fig:b_1.5_sequence}.

\

In Figure \ref{fig:b-positive-N} we show that the behaviour of the
discrete system is reproduced for the nonlinear system. It is
determined by the value of the initial firing rate
$$N_0:=-\p_vp_0(V_F),$$ depending on the position with respect to
$N_1^*$ and $N_2^*$.  The threshold value between both regimes is
$ N_2^*\approx2.294 $: for a lower value of $ N_0 $ the nonlinear
system is seen to converge to $ N_1^* $, while for $N_0>N_2^*$, the
firing rate seems to diverge.
\vspace*{1cm}
\begin{figure}[H]
	\begin{center}
		\begin{minipage}[c]{0.49\linewidth}
			\includegraphics[width=\textwidth]{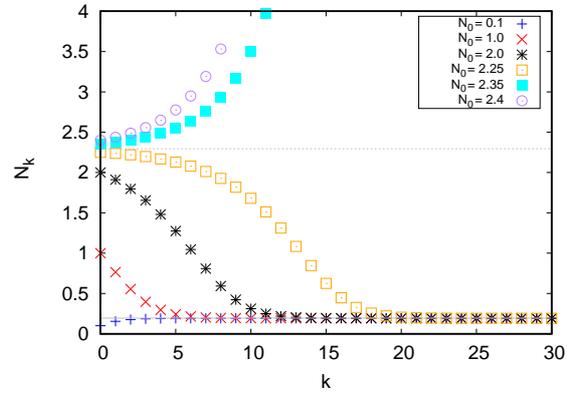}
		\end{minipage}
	\end{center}
	\caption{{\bf Firing rate sequence  $\boldsymbol{N_{k,\infty}}$
            \eqref{eq:sequence-N} with
            $\boldsymbol{b=1.5}$ and different values of
            initial condition $\boldsymbol{N_{0,\infty}}$.} Solid and
          dashed straight horizontal lines represent equilibria $ N_1^* $ and $ N_2^* $ respectively.
	}
	\label{fig:b_1.5_sequence}
\end{figure}
\newpage
In Figure \ref{fig:b-positive-N} we consider two different initial
conditions of the type \eqref{eq:profiles}, with the respective values
of $ N_0 = 2.233348 $ and $N_0 = 2.365824$, which serve as
representatives of two different regimes for the delayed nonlinear
system: convergence to the low steady state or formation of the
plateau distribution, when considering transmission delay $ d=10 $. This
illustrates bi-stability between the low steady state and the plateau
distribution, depending only on the initial condition (see \cite{CR-L}
for a better understanding of the plateau distribution). As we also
mention in \cite{CR-L}, no matter what the size of the delay is, we
shall find that the nonlinear system seems to behave according to the
pseudo-equilibria sequence. This appears to indicate that the
long-term behavior of the system can be decided only on the basis of
the value of $N_0$.
In the upper graphs, 
we see the time evolution of the firing rates and on the
bottom, 
we see the shape of the voltage distributions at the end of the two
simulations ($t=220$). When the system starts with an initial
condition whose $N_0$ is less than $ N_2^* $, the firing rate
convergences to $ N_1^* $ (left plot).  However, if the system starts
with $N_2^*<N_0$, then $ N(t) $ increases in time (right plot). This
behaviour is consistent with that of the firing rate and
pseudo-equilibria sequences described in Theorem
\ref{th:relations-Nb-positive}; if $N_{0,\infty}=N_0$ is below the
high stationary firing rate $ N_2^* $ then
$ \left\{N_{k,\infty}\right\}_{k\ge0} $ converges to the lower one
$ N_1^* $, and, if $N_{0,\infty}=N_0$ is higher than $ N^*_2 $ then
$ \left\{N_{k,\infty}\right\}_{k\ge0} $ diverges.

\begin{figure}[H]
	\begin{center}
	\begin{minipage}[c]{0.49\linewidth}
		\includegraphics[width=\textwidth]{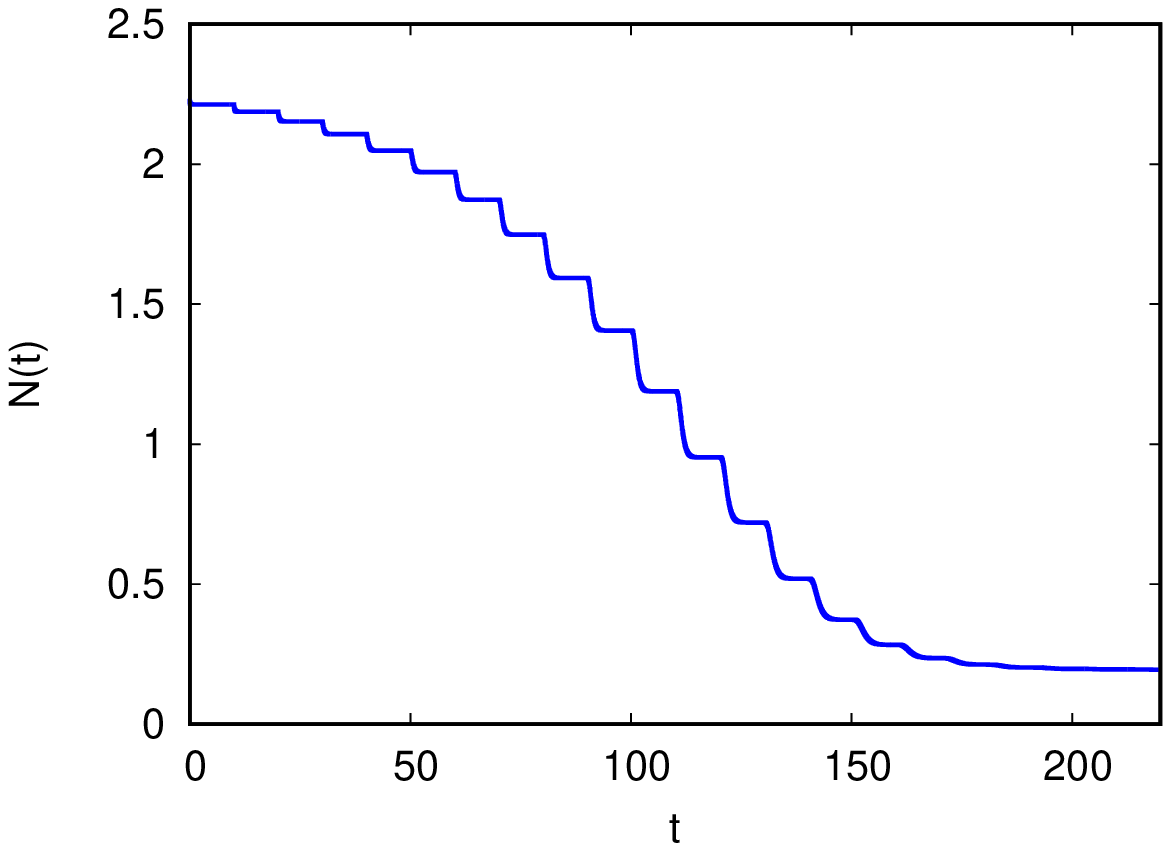}
	\end{minipage}    
	\begin{minipage}[c]{0.49\linewidth}
		\includegraphics[width=\textwidth]{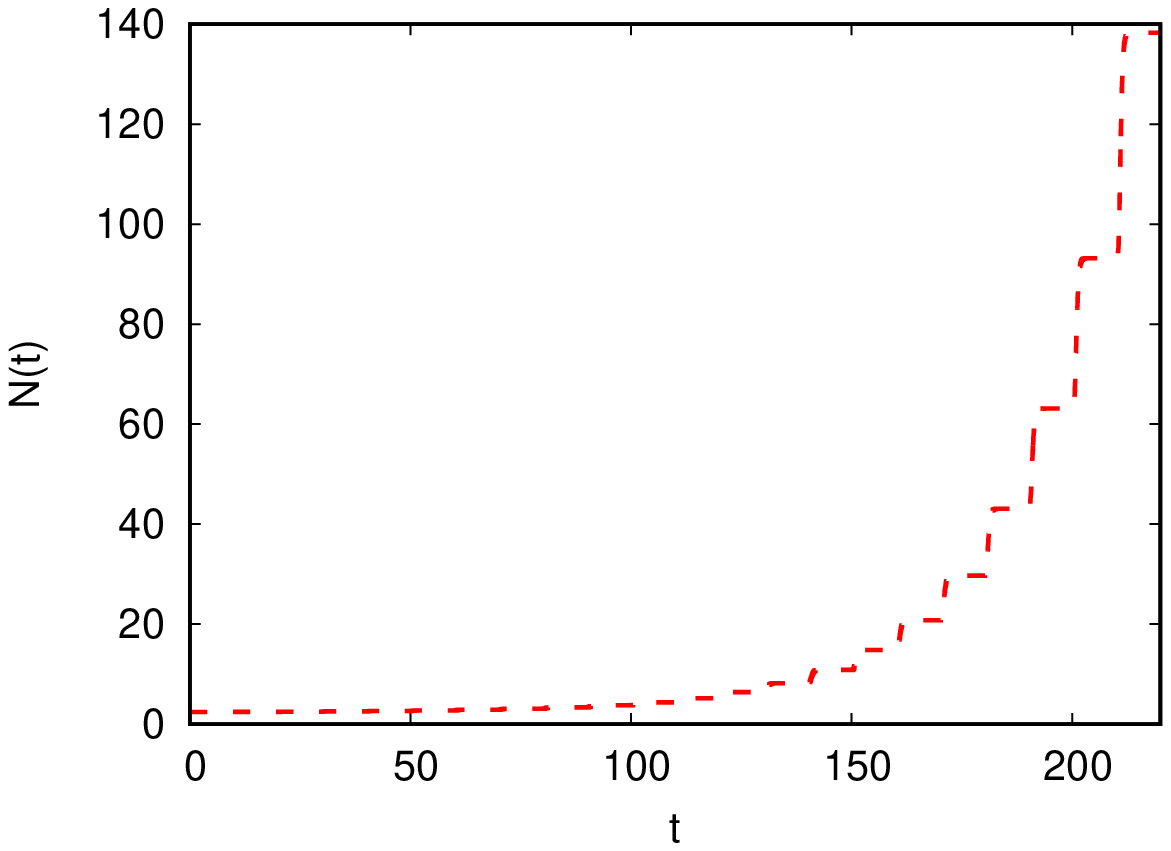}
	\end{minipage}
	\end{center}
	\begin{center}
		\begin{minipage}[c]{0.49\linewidth}
			\includegraphics[width=\textwidth]{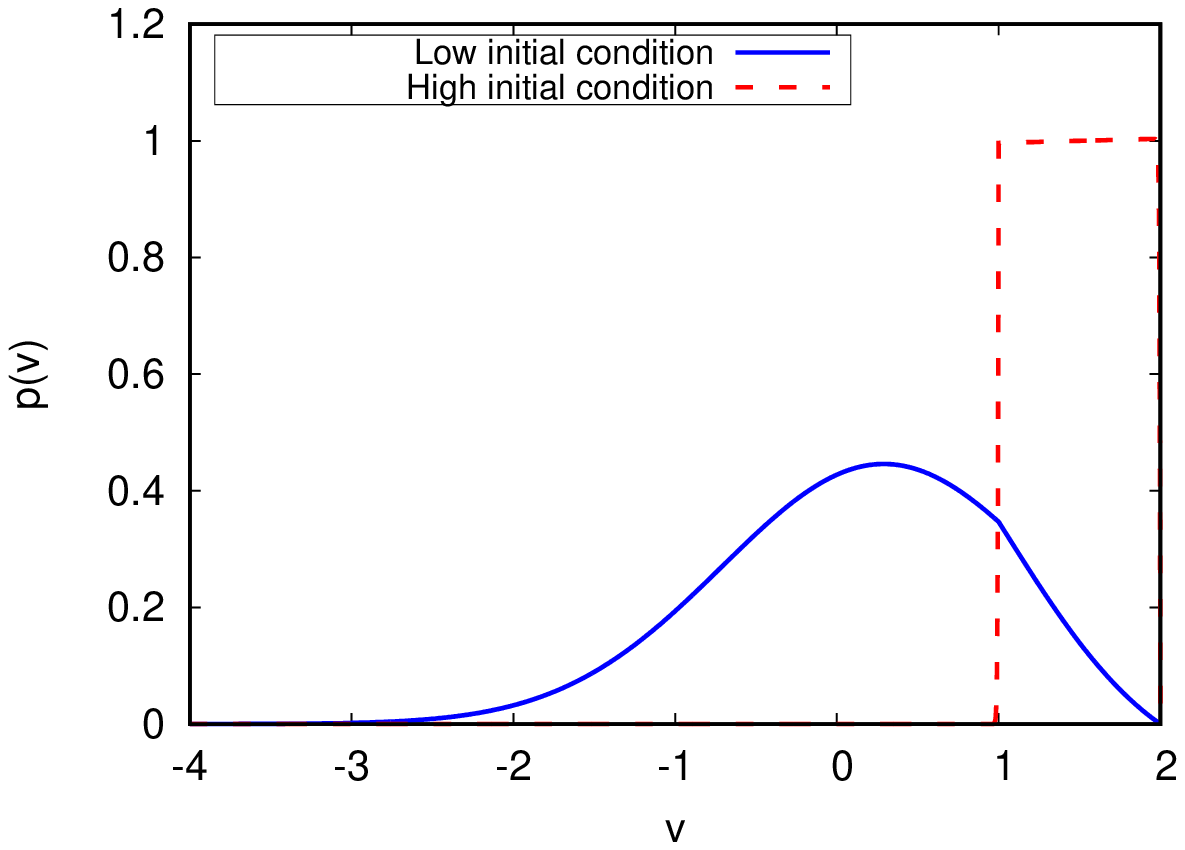}
		\end{minipage}
	\end{center}
	\caption{{\bf Nonlinear system \eqref{eq:edp-delay}
			with $\boldsymbol{b=1.5}$ and $\boldsymbol{d=10}$}.\newline
                      \emph{Top:} Evolution on time of the firing rate, $N(t)$, with initial conditions given by equation \eqref{eq:profiles} with $ N=2.25 $, $ \overline{N}=2.233348 $, chosen to be smaller than $ N_2^* $   (left)
                      and $ N=2.35 $, $ \overline{N}=2.365824 $, greater than $ N_2^* $ (right). 
\newline                      
                      \emph{Bottom:} Comparison of distributions $ p(v,t) $ at the end of the simulations ($ t=220 $).
	}
	\label{fig:b-positive-N}
\end{figure}

We have compared the discrete and the nonlinear systems in Figure
\ref{fig: b-positive-Nk-and-pk}. Here we have used the same
simulations shown in Figure \ref{fig:b-positive-N} for the nonlinear
system, and we have calculated the sequences $ \{N_{k,\infty}\}_{k} $
and $ \{p_{k,\infty}\}_k $ by starting with the same initial condition
as for the simulations, $ N_{0,\infty}=N_0=2.233348 $ in the left
plots and $ N_{0,\infty}=N_0=2.365824 $ in the right ones. In the top
plots we observe the comparison between $ N(t) $ and
$ \{N_{k,\infty}\}_{k} $ until $ t=200 $, while in the bottom plots we
show the comparison between $ p(v,t=20) $, $ p(v,t=120) $ and the
pseudo-equilibria $ p(v,t=200) $ with $ p_{2,\infty}(v) $,
$ p_{12,\infty}(v) $ and $ p_{20,\infty}(v) $ respectively. We can
observe that the delay $d=10$ is large enough so that the simulation
results, both $p(v,t)$ and $N(t)$, almost completely coincide with the
elements of the pseudo equilibrium and firing rate sequences starting
from the same initial condition.

\

In \cite{caceres2019global} a global existence theory was developed
for the nonlinear system \eqref{eq:edp-delay} with a transmission
delay $d>0$, by extending the results of \cite{carrillo2013classical}.
So the sequence of pseudo-equilibria suggests that the firing rate of
the nonlinear system should diverge and the theory tells us that this
cannot happen in finite time. The only possibility is then that $N(t)$
diverges in infinite time, giving rise to the plateau distribution as
$N(t)$ grows. This is precisely the behaviour seen in our simulations.

So that in cases where the firing rate sequence diverges, we expect
the following to be true: {\em Let us consider
	$0<b$, 
	an initial condition $p_0\in X$ 
	(and $N(t)=-\p_vp_0(V_F)$ for $t\in [-d,0]$) and $p$ its related
	solution to the nonlinear system \eqref{eq:edp-delay}.  Assuming
	that the firing rate sequence $\left\{N_{k,\infty}\right\}_{k\ge0}$,
	with initial value $N_{0,\infty}:=-\p_vp_0(V_F)$ (see
	\eqref{eq:sequence-N}) diverges, then the solution, $p$, to the
	nonlinear system \eqref{eq:edp-delay} with transmission delay $ d>0$
	and initial condition $p_0$ evolves to a plateau distribution, i.e,
	the membrane potential of the system tends to be uniformly
	distributed between $V_R$ and $V_F$.}

\

The numerical results in Figures \ref{fig:b_1.5_sequence} and
\ref{fig:b-positive-N} and numerical experiments in the literature
also illustrate the behaviour of the nonlinear system when the firing
rate sequence converges.  Theorems \ref{th:relations-Nb-positive} and
\ref{th:relations-Nb-negative} give conditions on the initial value of
the firing rate sequence to converge to an equilibrium of the system
(in both cases, excitatory and inhibitory).  In Figures
\ref{fig:b-positive-N} and \ref{fig: b-positive-Nk-and-pk} we see that
the nonlinear system tends to an equilibrium when the discrete system
also tends to equilibrium. In the following experiments below we shall
see it for the inhibitory case and in
\cite{CR-L,caceres2011analysis,caceres2018analysis,hu2021structure} it
can also be seen for the excitatory case with only one equilibrium.

So that in case where firing rate sequence converges,
we  expect the following to be true:
{\em	Let us  consider $ b\in\R $, 
	an initial condition $ p_0\in X $ 
	(and $N(t)=-\p_vp(V_F)$ for $t\in [-d,0]$) and  $p$ its related solution
	to the nonlinear system \eqref{eq:edp-delay}.
	Assuming that  
	the firing rate sequence $\left\{N_{k,\infty}\right\}_{k\ge0}$, with initial value
	$N_{0,\infty}:=-\p_vp_0(V_F)$
	(see \eqref{eq:sequence-N}) converges to $ N_\infty>0 $,
	then there exists $ d_0>0 $ large enough and $ Q,\mu>0 $, such that the solution,
	$p$,  to the nonlinear system \eqref{eq:edp-delay} 
	with transmission delay  $ d>d_0 $ and initial condition $ p_0 $ fulfills
	\begin{equation}
		\|p(.,t)-p_\infty(.)\|_X\le Qe^{-\mu t}\|p_0-p_\infty\|_X \quad \forall t\ge 0,
	\end{equation}
	where 
	$p_\infty$  is the stationary solution to the nonlinear system
	\eqref{eq:edp-delay} with firing rate $N_\infty$.
}

\begin{figure}[H]
	\begin{center}
		\begin{minipage}[c]{0.49\linewidth}
			\includegraphics[width=\textwidth]{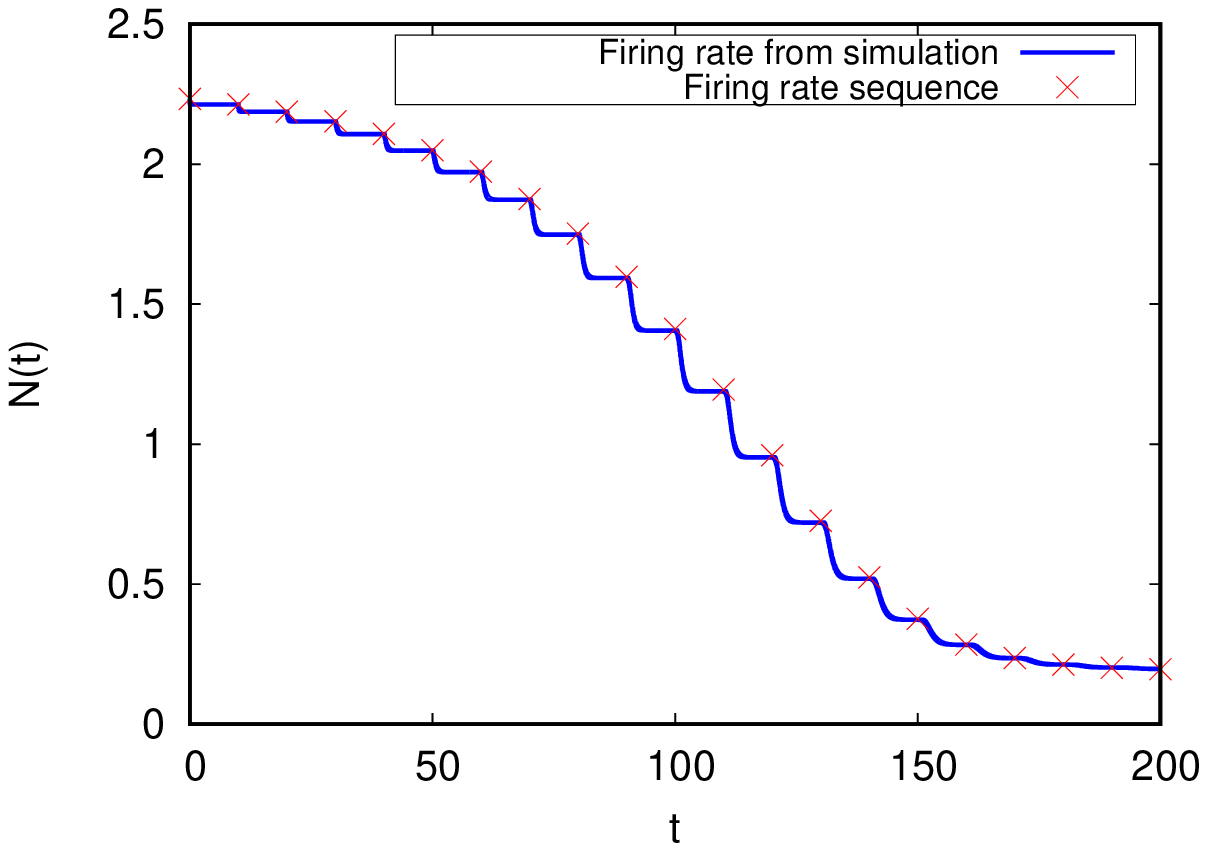}
		\end{minipage}    
		\begin{minipage}[c]{0.49\linewidth}
			\includegraphics[width=\textwidth]{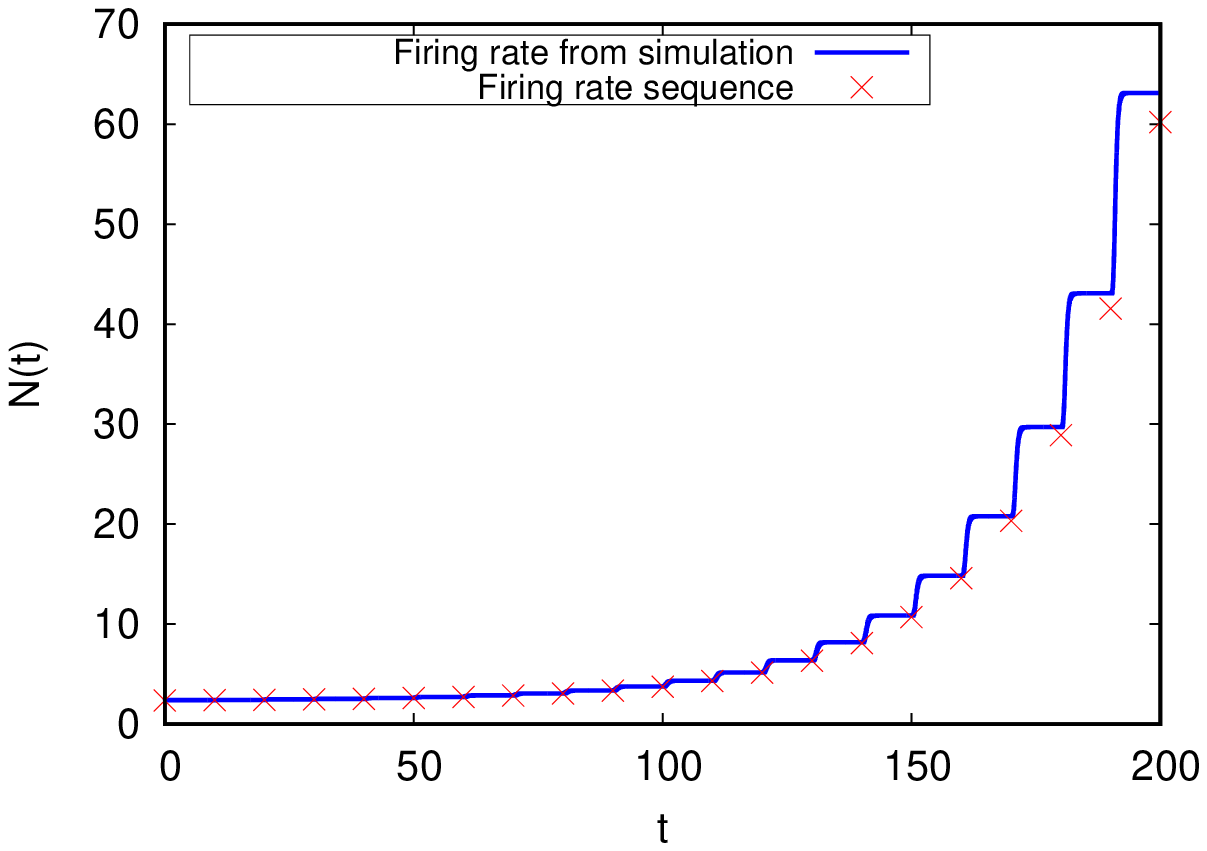}
		\end{minipage}
	\end{center}
	\begin{center}
		\begin{minipage}[c]{0.49\linewidth}
			\includegraphics[width=\textwidth]{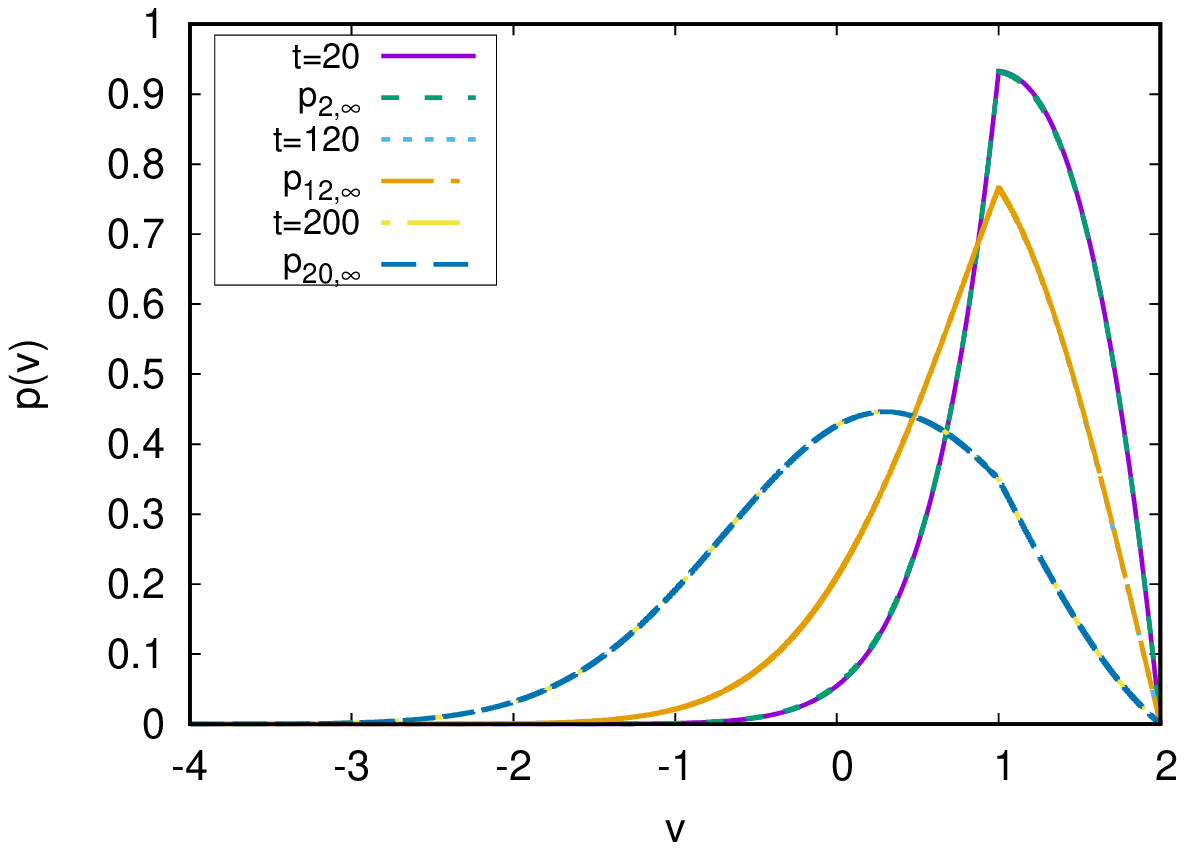}
		\end{minipage}
		\begin{minipage}[c]{0.49\linewidth}
			\includegraphics[width=\textwidth]{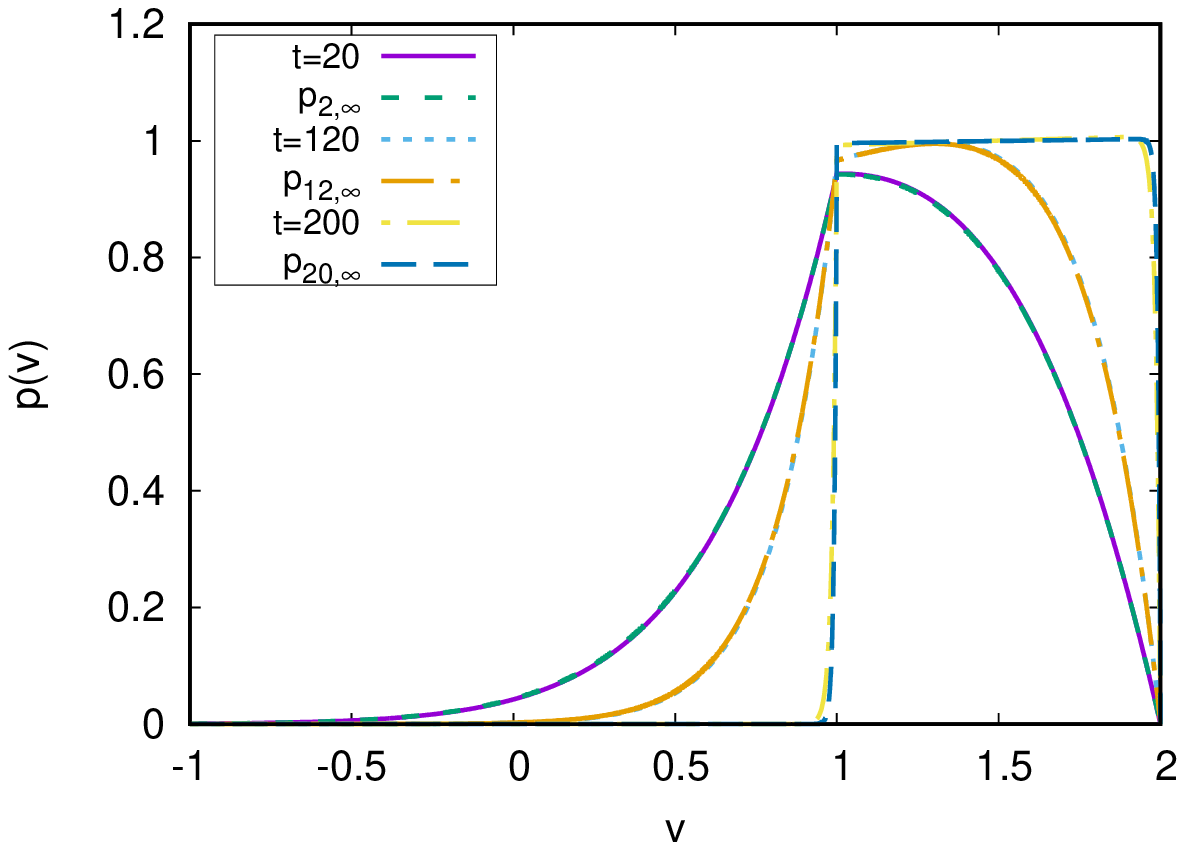}
		\end{minipage}
	\end{center}
	\caption{{\bf Comparison between the nonlinear system
            \eqref{eq:edp-delay} and the discrete system with
            $\boldsymbol{b=1.5}$ and $\boldsymbol{d=10}$.}
          \emph{Top:} Comparison between the firing rate of the
          approximated solution to the nonlinear equation, $ N(t) $
          and the sequence of firing rates $ N_{k,\infty}
          $. \emph{Bottom:} Solutions of the nonlinear problem,
          $ p(v,t) $ at different times, compared with the
          pseudo-equilibria $ p_{k,\infty}(v) $ which correspond to
          those times. \emph{Left:} initial condition given by
          equation \eqref{eq:profiles} with $ N=2.25 $,
          $ \overline{N}=2.233348 $, chosen to be smaller than
          $ N_2^* $. \emph{Right:} initial condition given by equation
          \eqref{eq:profiles} with $ N=2.35 $,
          $ \overline{N}=2.365824 $, greater than $ N_2^* $.  }
	\label{fig: b-positive-Nk-and-pk}
\end{figure}

\subsubsection*{Influence of the delay value on the excitatory nonlinear system behavior: cases with one equilibrium and without equilibria}

In this subsection we show results concerning the influence of the delay value on the behavior of the nonlinear system. As we said before, the nonlinear system must follow the behavior of the pseudo-equilibrium sequence when the delay is sufficiently large. Although deciding when the delay is large enough depends on the parameters of each simulation (especially depends on $b$ and the initial condition of the firing rate $N_0$). To evaluate the influence of the delay on the nonlinear system we will use two situations where at least numerically the behavior of the nonlinear system is well known. The first is the case with $b=0.5$ and therefore a single equilibrium, where we know that, if we consider any delay $d>0$ the system converges to its unique equilibrium. The second is the case $b=2.2$, where there are no equilibria and we know that any system with delay tends to a plateau distribution, i.e., its firing rate grows in time but does not diverge in finite time.

Therefore we will see how the use of different values for the delay $d$ does not change the fundamental long-term behavior of the system, but makes the system in each delay interval converges to the corresponding pseudo-equilibrium or not. Considering that there is only one possible long-term behavior for each case, we will consider a single initial condition in each case and different values for the delay.

First, in Figure \ref{fig:b_05_22_sequences} we observe the sequence of firing
rates for both values of $b$, which behave as mentioned before for the
nonlinear system with delay: for $b=0.5$, it converges to the unique
equilibrium regardless of the initial condition; for $b=2.2$, it increases
in time.
In Figure \ref{fig:b_05_22_nonlinear_system} we show the evolution on time of
the firing rate of the nonlinear system, $N(t)$, with
three different values of the delay.
In the left plot, considering $b=0.5$, we note the need for a large delay
(at least greater than $2$) for the nonlinear system to pass one by one through the pseudo-equilibria and even remain close to them for some time.
However we note from the right plot of Figure \ref{fig:b_05_22_nonlinear_system} the lower requirement of high delay values to observe the nonlinear
system stabilizing for a certain time in the pseudo-equilibria,
the necessary value being somewhere between 0.1 and 1.

If we consider these results together with those shown in the previous
subsection for $b=1.5$, we notice an influence of the value of $b$ on the
value of $d$ necessary for the system to pass through its associated
pseudo-equilibria. So that the smaller the value of $b$ is, the greater the
value of $d$ should be. This consideration is in accordance with the stated
in the proof of  Theorem \ref{thm:convergence}
(see \eqref{eq:requirement-delay}).

\begin{figure}[H]
	\begin{center}
		\begin{minipage}[c]{0.49\linewidth}
			\includegraphics[width=\textwidth]{b_05_sequence.eps}
		\end{minipage}    
		\begin{minipage}[c]{0.49\linewidth}
			\includegraphics[width=\textwidth]{b_22_sequence.eps}
		\end{minipage}
	\end{center}
	\caption{{\bf Firing rate sequence  $\boldsymbol{N_{k,\infty}}$
	    \eqref{eq:sequence-N} with 
            $\boldsymbol{b=0.5}$ and  $\boldsymbol{b=2.2}$.}\\
          \emph{Left:} $b=0.5$ and different values of
		initial condition $N_{0,\infty}$ below and above the unique equilibrium (gray dashed line). \emph{Right: } $b=2.2$. This is a case without equilibria.
	}
	\label{fig:b_05_22_sequences}
\end{figure}

\begin{figure}[H]
	\begin{center}
		\begin{minipage}[c]{0.49\linewidth}
			\includegraphics[width=\textwidth]{b05_high_N.eps}
		\end{minipage}    
		\begin{minipage}[c]{0.49\linewidth}
			\includegraphics[width=\textwidth]{b22_N.eps}
		\end{minipage}
	\end{center}
	\caption{{\bf  Time evolution of the firing rate, $\boldsymbol{N(t)}$,
            for the nonlinear system \eqref{eq:edp-delay} with different
            values of the delay and the connectivity parameter.}
          \newline {\em Left:} $b=0.5$ and different values of
		delay.  Initial conditions are given by equation \eqref{eq:profiles} with $ N=6 $. In this case there is unique equilibrium.
		\emph{Right:} $b=2.2$ and different value of the delay.
                Initial conditions are given by equation \eqref{eq:profiles}
                with $ N=2 $. For each delay value $d$,
                the plot shows the firing rate in a time range $[0,3d]$.
                This is a case without equilibria. In both plots dashed lines correspond with values of the firing rate sequence \eqref{eq:sequence-N} with $b$ and $N_{0,\infty}$ (initial condition) given by the parameters of the nonlinear system.
	}
	\label{fig:b_05_22_nonlinear_system}
\end{figure}

\subsubsection*{Periodic states in the highly inhibitory case with large delay}

In this second experiment we discover conditions under which periodic
solutions appear in an inhibitory system. This phenomenon was observed numerically in \cite{ikeda2022theoretical}. We have determined the value of parameter $ b $ for which the periodic states appear for large delay, in the light of the behaviour of the sequence of pseudo-equilibria. Here we find an important
difference with the excitatory case, where the size of the delay is
not essential to determine whether the long-term behavior of the
system conforms to that described by the pseudo equilibrium
sequence. For highly inhibitory systems the size of the transmission
delay is key.  We recall that for the inhibitory case the nonlinear
system has only one steady state $p_\infty$, with its associated
stationary firing rate $N_\infty$, and the behaviour of the firing
rate sequence $N_{k,\infty}$ (the discrete system) is given by Theorem
\ref{th:relations-Nb-negative}: $N_{k,\infty}$ tends to the
equilibrium or to 2-cycle.

\newpage

Our first aim is to find appropriate values of the parameters for
which periodic states emerge.  Then we shall focus on the study of
periodic solutions.  Figure \ref{fig:b_star} suggests which values we
should choose for $b$ in order to find periodic states.  We see that
the bifurcation value $b^*$ is around the value $b^*\approx -9.4$ (in
the sense of Theorem \ref{th:relations-Nb-negative}).  Thus, the
firing rate sequence $N_{k,\infty}$ converges to 
$N_\infty$ if $b^*<b$ and tends towards a 2-cycle
$\left\{N^-,N^+\right\}$, if $b<b^*$.  This tells us that a
sufficiently large delay will allow us to observe the same behavior
when simulating the nonlinear system.  We must emphasize that for all
cases with $b\le0$, the choice of initial condition $N_{0,\infty}$
does not appear to change the long-term values of the firing rate
sequence $N_{k,\infty}$.

\begin{figure}[H]
	\begin{center}
		\begin{minipage}[c]{0.49\linewidth}
			\includegraphics[width=\textwidth]{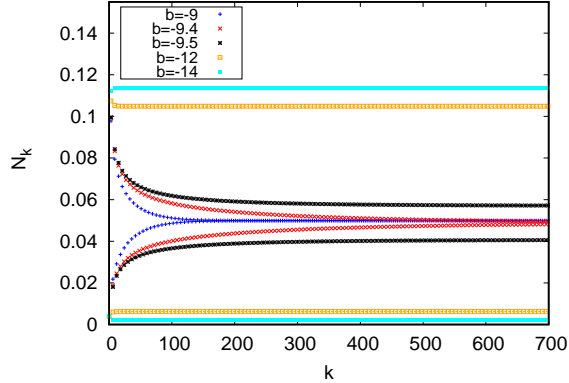}
		\end{minipage}
	\end{center}
	\caption{{\bf Firing rate sequence $\boldsymbol{N_{k,\infty}}$
            \eqref{eq:sequence-N} for different values of
            $ \boldsymbol{b <0}$ with
            $\boldsymbol{N_{0,\infty}=0.004} $. For $b < -9,4$ the
            sequence tends towards a 2-cycle.}  }
	\label{fig:b_star}
\end{figure}

Due to the numerical study of the firing rate sequence, we have a good
guess for the values of $b$ needed to find periodic states. Let us
test different values of $b$ and $d$ and observe the behavior of the
nonlinear system. The following simulations have been performed using
as initial condition the profile given by Equation
\eqref{eq:profiles}, with $N=0$, letting the system evolve up to
$t=300$.  In the left plot of Figure \ref{fig:b-positive-N-comp} we
show the firing rate $ N(t) $ of the nonlinear system, with delay
$ d=10 $, testing three different values of $ b $ to find stable
oscillations.  For $ b $ not sufficiently negative, the amplitude of
the oscillations is observed to be damped in time, suggesting a
tendency towards the single steady state of the system, as we showed
earlier in the study of the sequence of pseudo-equilibria (see Theorem
\ref{th:relations-Nb-negative}). For a sufficiently negative value of
$b$, as shown for $b=-12$, we observe fairly stable oscillations in
time.  Thus, to be sure of the stability of the solutions without the
need for an excessively large delay, we will choose $b=-14$ for the
following experiments.  In the right plot of Figure
\ref{fig:b-positive-N-comp} we can see, for $b=-14$, how the choice of
a large enough delay permits the system to evolve towards a periodic
state.  This picture shows the evolution in time of the firing rate
$N(t)$ with different delay values $d$. If $d$ is small ($d=2$) the
firing rates tends to a stationary value, while if $d$ is large
($d=10$, $d=25$), its behaviour tends to be periodic.  We must
emphasize that, when we do not set a sufficiently high delay, even if
the value of $b$ is very negative and therefore the pseudo-equilibria
sequence in Theorem \ref{th:relations-Nb-negative} points to periodic
behavior, we shall observe convergence to the steady state, as it is
shown with solid line in the right graph, for $d=2$.  This is a strong
difference between the inhibitory case with periodic state and the
excitatory case with plateau state, where the size of the delay was
irrelevant.

\begin{figure}[H]
	\begin{center}
		\begin{minipage}[c]{0.49\linewidth}
			\includegraphics[width=\textwidth]{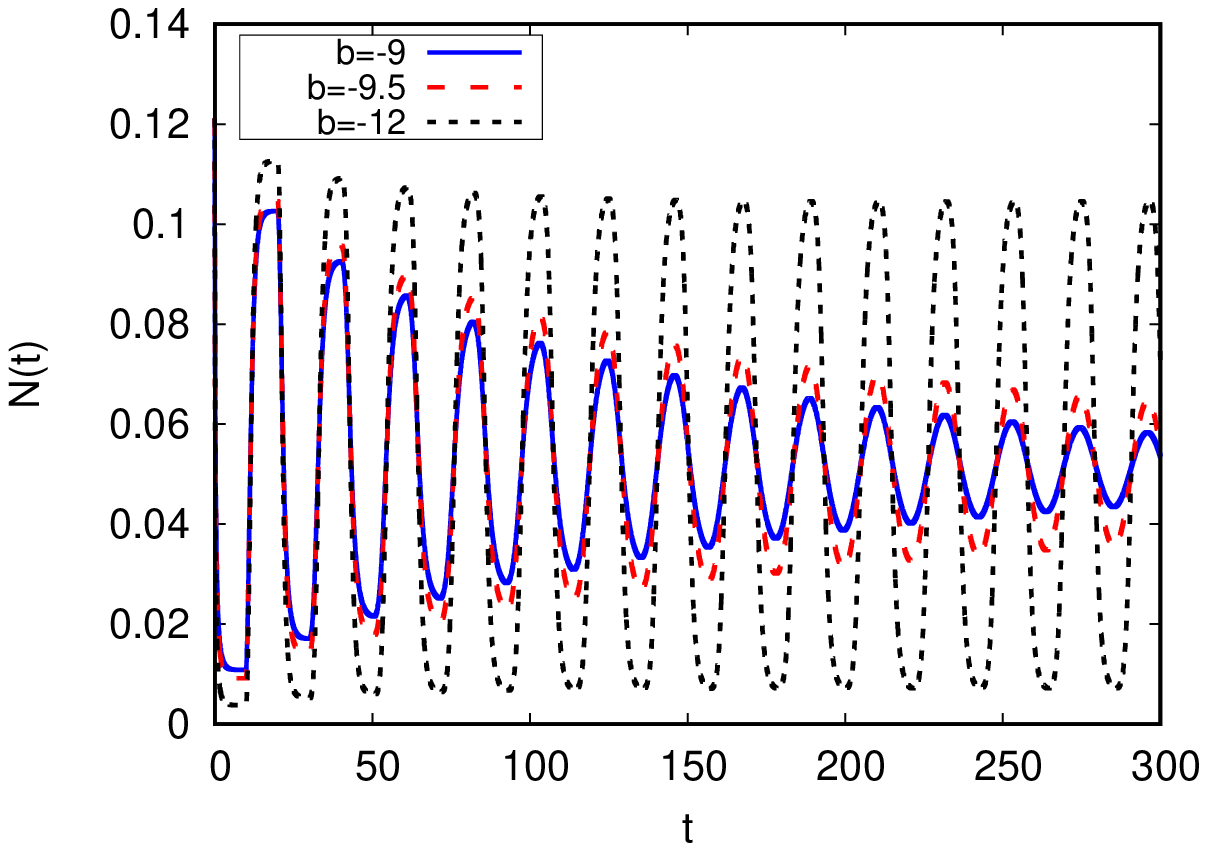}
		\end{minipage}    
		\begin{minipage}[c]{0.49\linewidth}
			\includegraphics[width=\textwidth]{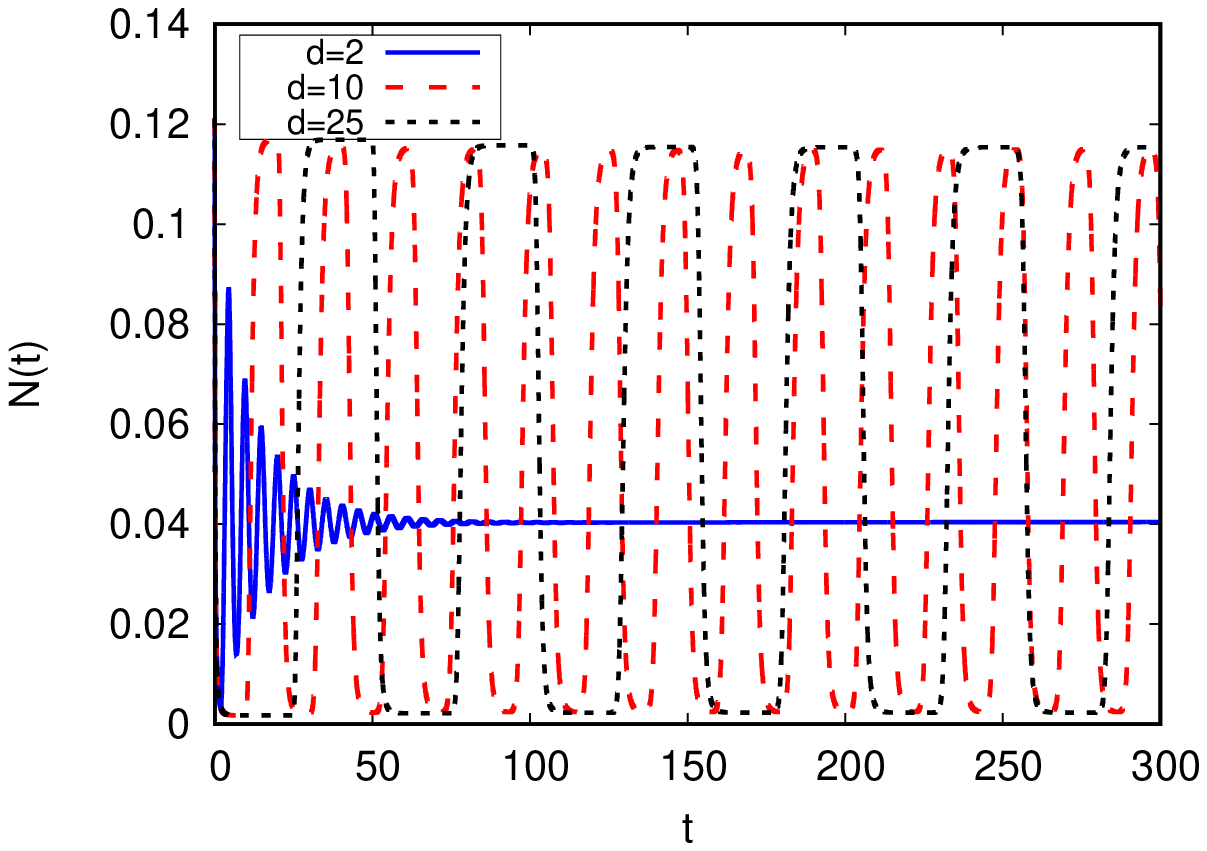}
		\end{minipage}
	\end{center}
	\caption{{\bf Nonlinear system \eqref{eq:edp-delay}} Time
          evolution of the firing rate, $N(t)$. The initial condition
          is given by equation \eqref{eq:profiles} with
          $ N=0 $.\newline {\em Left:} Different values of the
          connectivity parameter $b$ with delay $d=10$.  {\em Right:}
          Connectivity parameter $b=-14$ with different values of the
          delay.  }
	\label{fig:b-positive-N-comp}
\end{figure}

To study the periodic solutions of the highly inhibitory nonlinear system, 
we shall set $b=-14$ and $d=25$ for the rest of the simulations.

For these parameters,
the steady state of 
the nonlinear system $ p_\infty(v) $ has firing rate $N_\infty=0.0396$.
When we study numerically the firing rate sequence $N_{k,\infty}$
we find a tendency towards the values $ N^-= 0.0022$ and $ N^+=0.1136 $ in the sense of Theorem \ref{th:relations-Nb-negative}, as shown above in Figure \ref{fig:b_star} for $b=-14$.

\

In Figure \ref{fig: b-negative-long-time} we compare the solution to
the nonlinear system \eqref{eq:edp-delay} with the 2-cycle,
$\{p^-(v),p^+(v)\}$, found in Theorems \ref{th:relations-Nb-negative}
and \ref{thm:two_cicle_convergence} for the succession of
pseudo-equilibria $ \left\{p_{k,\infty}\right\}_{k\ge0} $.  In the
left graph we observe the comparison between the lower state
$ p^-(v) $ and the approximated solution to the Fokker-Planck
equation, starting with an initial condition given by an approximation
of $ p^-(v) $. The times shown in the graphs have been selected so
that the distribution should coincide with $p^-(v)$, since they are
even multiples of the delay. In the graph on the right we see a
similar comparison between $ p^+(v) $ and the approximated $ p(v) $,
starting now with an approximation of $ p^+(v) $, also in the times
they should overlap.  We can see that in the left graph there
is 
a small difference between the values of $p(v,t)$ for the different
times represented with respect to $p^-(v)$, but in the right graph the
difference is practically negligible.  This may be due to the fact
that the observed period for the periodic state is not exactly $2d$,
as can be seen in the bottom plot, where we show the firing rate of
the system, computed for a simulation with initial condition given by
Equation \eqref{eq:profiles} with $ N=0 $.

\begin{figure}[H]
	\begin{center}
		\begin{minipage}[c]{0.49\linewidth}
			\begin{center}
				\includegraphics[width=\textwidth]{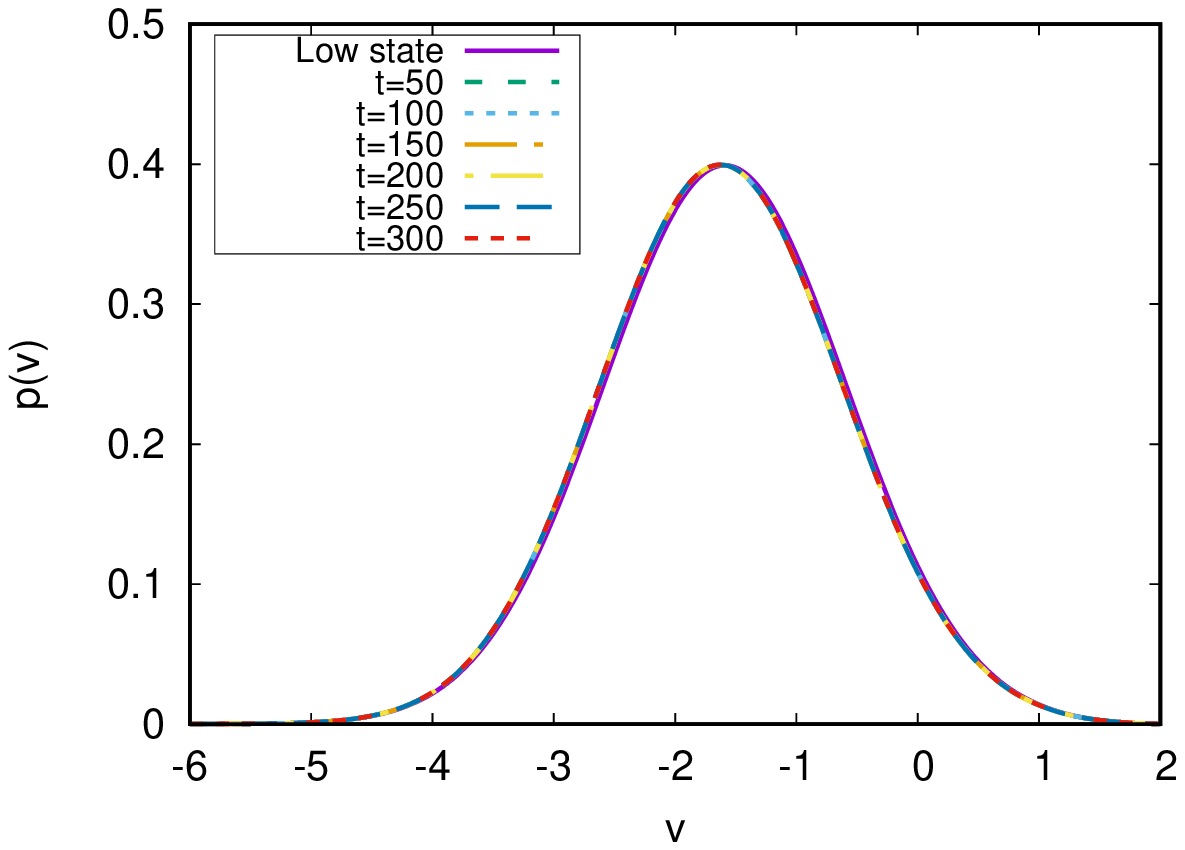}
			\end{center}
		\end{minipage}
		\begin{minipage}[c]{0.49\linewidth}
			\begin{center}
				\includegraphics[width=\textwidth]{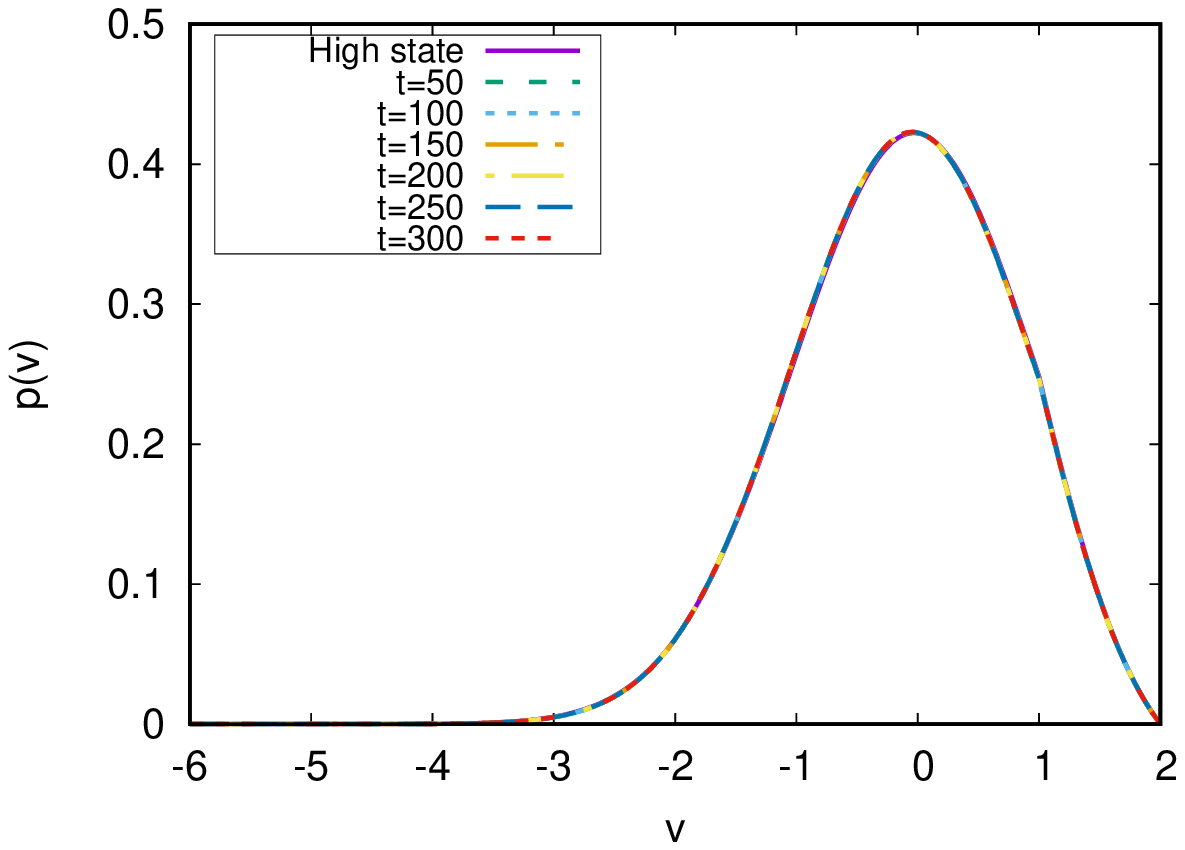}
			\end{center}
		\end{minipage}
	\end{center}
	\begin{center}
	\begin{minipage}[c]{0.49\linewidth}
		\includegraphics[width=\textwidth]{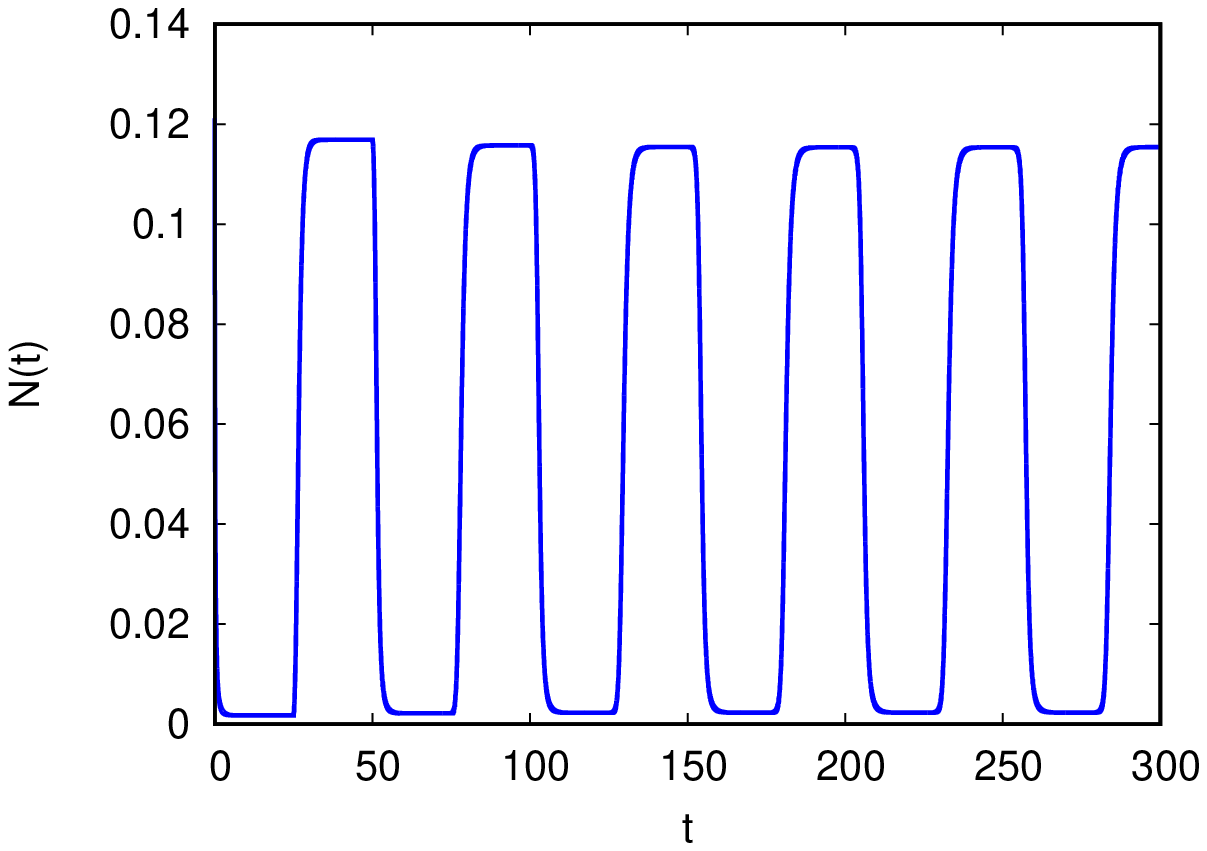}
	\end{minipage}
\end{center}
	\caption{{\bf Nonlinear system \eqref{eq:edp-delay}
			with $\boldsymbol{b=-14}$ and delay value
			$\boldsymbol{d=25}$}.\newline
		\emph{Top left:} Initial condition given by the pseudo equilibrium $ p^-(v) $, from equation \eqref{eq:profile-}. Approximated solution to the nonlinear system, $ p(v,t) $, at different times, compared with $ p^-(v) $. \emph{Top right:} Initial condition given by the pseudo equilibrium $ p^+(v) $, from equation \eqref{eq:profile+}. Approximated $ p(v,t) $ at different times, compared with $ p^+(v) $. \emph{Bottom:} Initial condition given by \eqref{eq:profiles} with $ N=0 $. Time evolution of the firing rate $ N(t) $.
	  }
	\label{fig: b-negative-long-time}
\end{figure}

Finally, in Figure \ref{fig: b-positive-N}, we analyse how the initial condition
influences the evolution of the nonlinear system. We consider
four different initial conditions shown in the top left graph:
$p^-(v)$, $ p^+(v) $ and two double Maxwellians distributions,
as \eqref{eq:2maxw}, with $\mu_{\text{low}}=-1$, $\mu_{\text{high}}=0.4$
and $ \sigma=0.5 $ in both cases.
Our purpose is to provide evidence that regardless of the initial condition, with these values of $b$ and $d$, the system will end up in a periodic state between the pseudo-equilibria $p^-(v)$ and $p^+(v)$, with pseudo-stationary firing rates $N^-$ and $N^+$. Nevertheless, it is worth noting how the initial condition determines which pseudo equilibrium is going to be reached first, determining the times at which the distribution $p(v,t)$ will be close to $p^-(v)$ and $p^+(v)$.

In top right plot of  Figure \ref{fig: b-positive-N},
we can see the evolution on time of the four firing rates
of the different simulations. We note that for $t>150$
the firing rates of the simulations starting with the low state and
the double Maxwellians low ($\mu_{\text{low}}=-1$) are synchronised, and
the same is true for the firing rates of the simulations starting with
the high state and the double Maxwellians high ($\mu_{\text{high}}=0.4$).
We also see, in the bottom graphs, these synchronies of the corresponding
distributions $ p(v,t) $, 
``passing'' through $ p^-(v) $ and $ p^+(v) $
at the same times ($t=250$ left and $t=264$ right), and describing,
in this way, a periodic behaviour.

\begin{figure}[H]
	\begin{center}
		\begin{minipage}[c]{0.49\linewidth}
			\includegraphics[width=\textwidth]{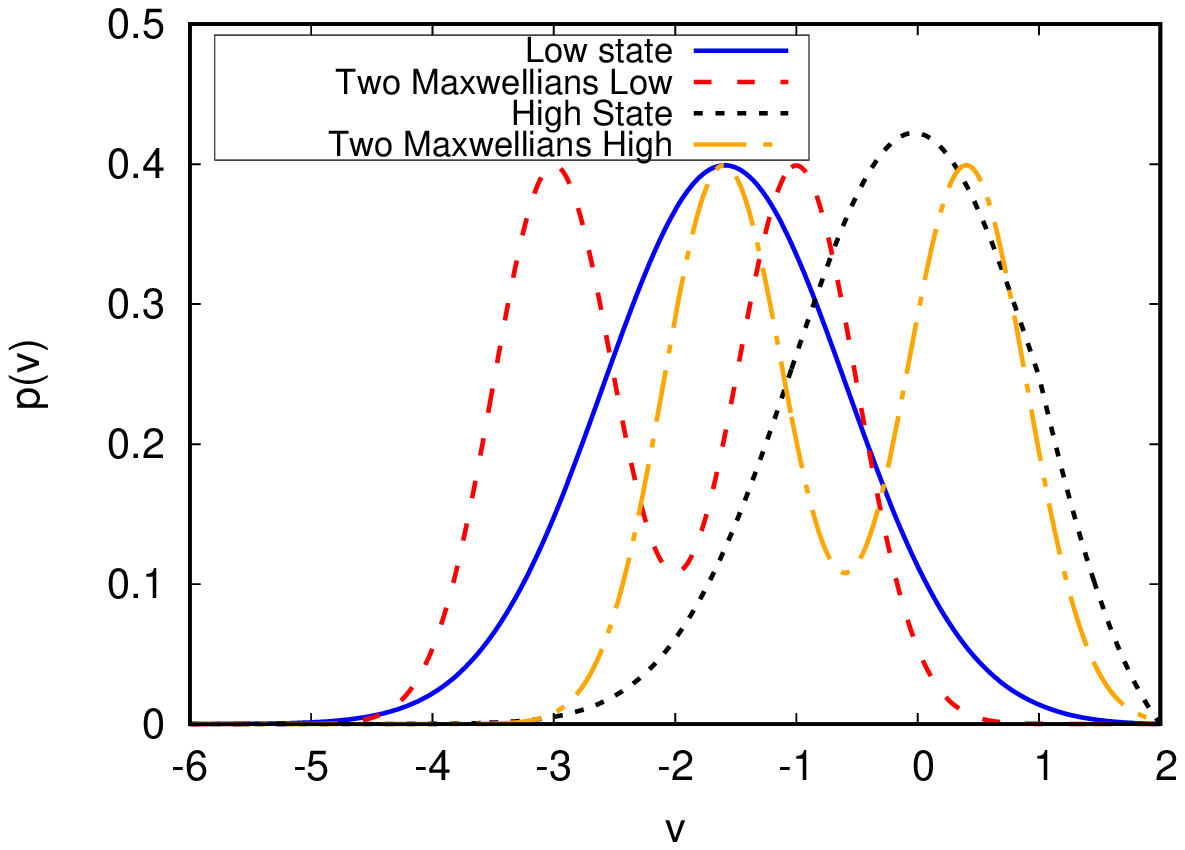}
		\end{minipage}    
		\begin{minipage}[c]{0.49\linewidth}
			\includegraphics[width=\textwidth]{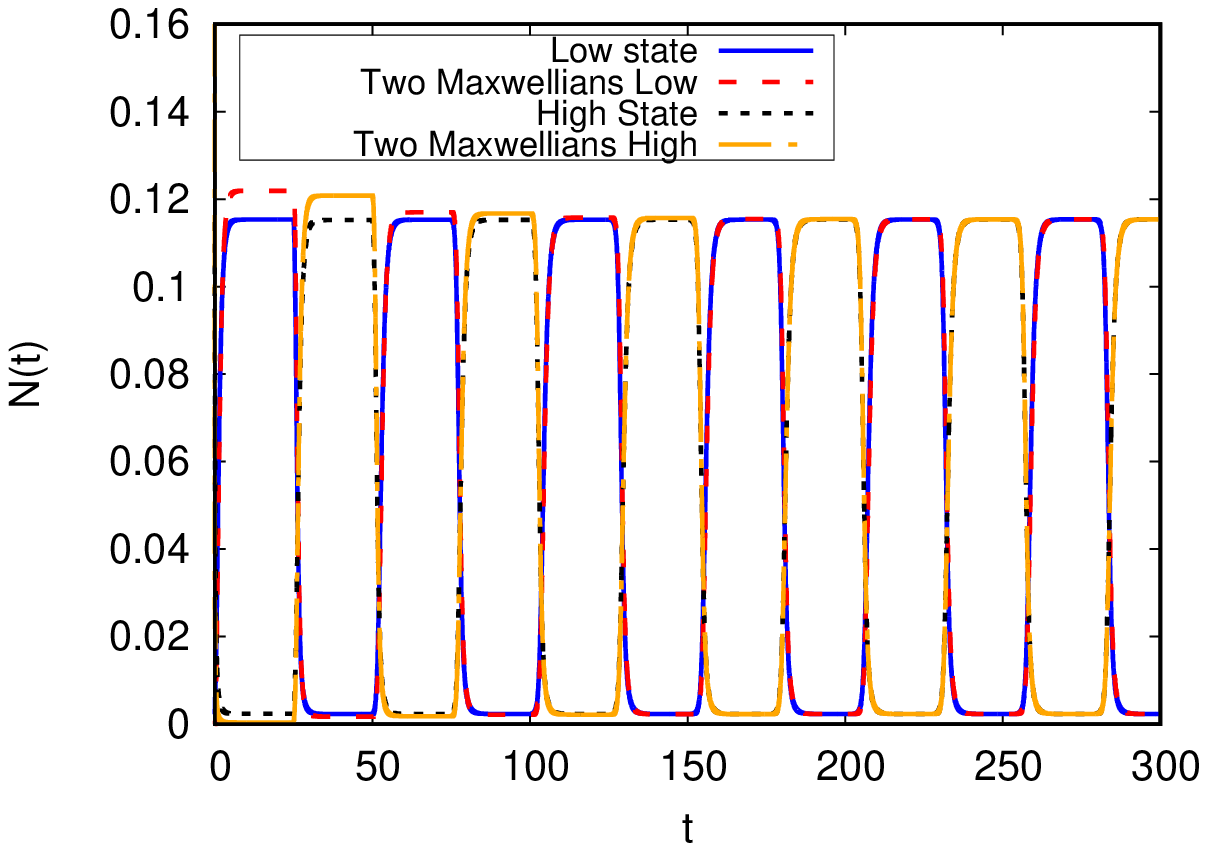}
		\end{minipage}
	\end{center}
	\begin{center}
		\begin{minipage}[c]{0.49\linewidth}
			\includegraphics[width=\textwidth]{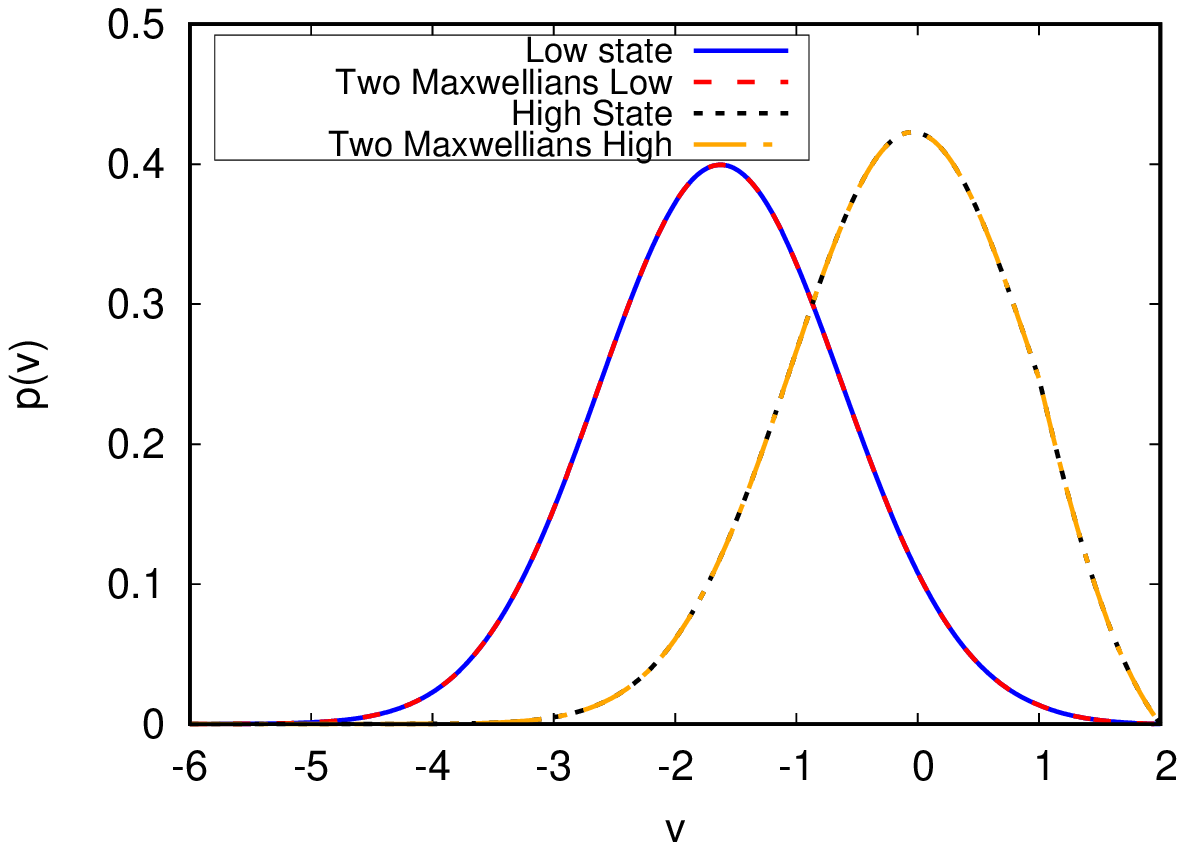}
		\end{minipage}
		\begin{minipage}[c]{0.49\linewidth}
			\includegraphics[width=\textwidth]{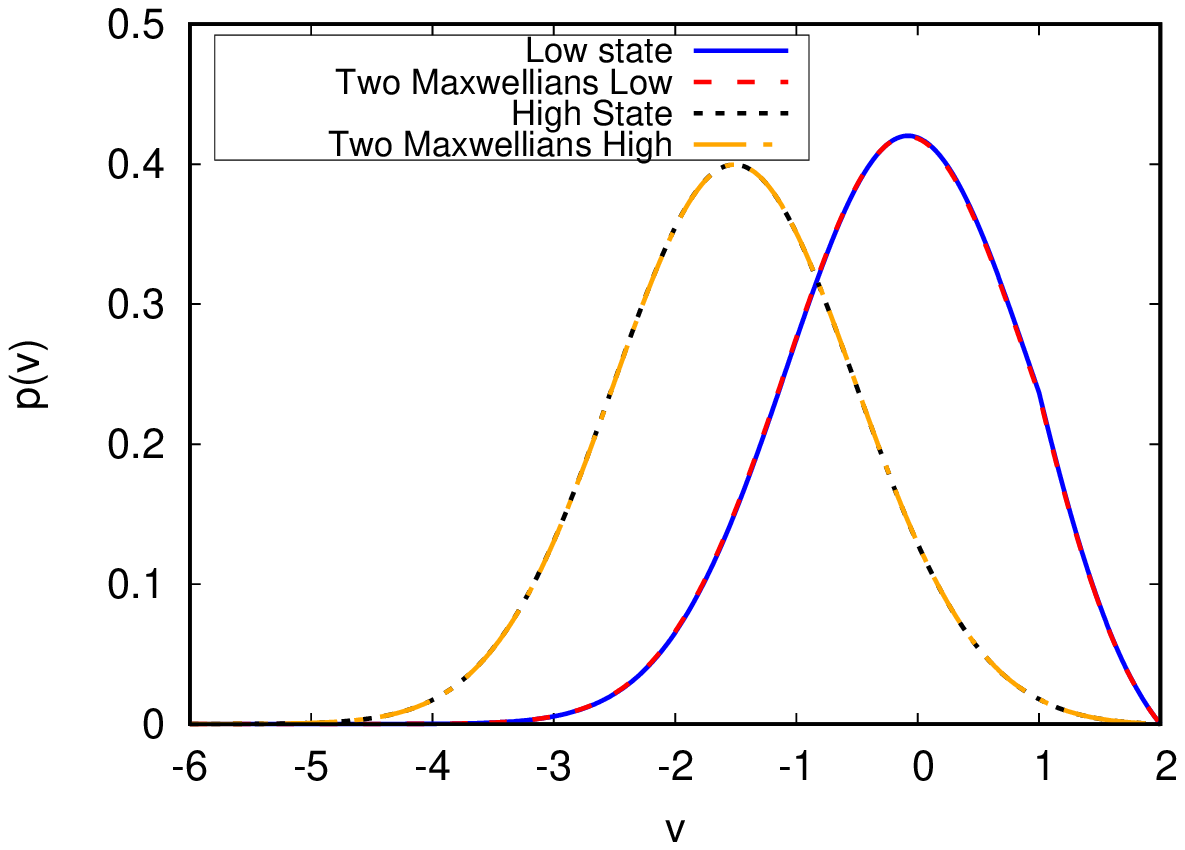}
		\end{minipage}
	\end{center}
	\caption{{\bf Nonlinear system \eqref{eq:edp-delay}
			with $\boldsymbol{b=-14}$ and $\boldsymbol{d=25}$} using four different initial conditions. \emph{Top left: } Initial conditions: we consider the low pseudo equilibrium given by equation \eqref{eq:profile-}, the high pseudo equilibrium given by equation \eqref{eq:profile+} and two configurations of double Maxwellians given by equation \eqref{eq:2maxw}, with $ \mu_{\text{low}}=-1 $, $ \mu_{\text{high}}=0.4 $ and $ \sigma=0.5 $ in both cases. \emph{Top right: } Comparison between the time evolution of the firing rate $ N(t) $ depending on the initial condition. \emph{Bottom left and right: } Distribution $ p(v) $ at times $ t=250 $ and $ t=264 $ respectively, for both initial conditions.
	}
	\label{fig: b-positive-N}
\end{figure}

The behaviours shown in Figures \ref{fig:b-positive-N-comp} and
\ref{fig: b-negative-long-time} 
agree with that of the firing rate and pseudo-equilibria sequences described in Theorems \ref{th:relations-Nb-negative} and \ref{thm:two_cicle_convergence}; if $b^*\approx9.5< b<0 $ we find convergence to the steady state, no matter what the initial condition is. If $ b<b^* $ then $ d $ should be sufficiently large to find a periodic state. Otherwise, there will be convergence to equilibrium. Finally, Figure \ref{fig: b-positive-N} helps us to understand which are the different options for the possible periodic states to which the nonlinear system tends, as well as their dependence on the initial condition. 

  The periodic states seem to be determined by the solutions starting from
  one of the two pseudo-equilibria. These solutions, when the delay is large,
  approach the other pseudo equilibrium at the end of each delay period.
  In the next delay period they return to the initial pseudo equilibrium,
  and so continue in the following periods of length $d$.
  And they tends to a periodic state, which we can call $p_\infty^-(v,t)$ and
  $p_\infty^+(v,t)$, depending if the initial condition is $p^-(v)$
  or $p^+(v)$, respectively.
  Simulations lead us to suspect that the period is $2d+\epsilon$,
  because of the time it takes to move from one pseudo equilibrium to
  another, and $p_\infty^-(v,t)=p_\infty^+(v,t+d+\epsilon/2)$.
  
To be more precise
what we would expect to be demonstrated in the nonlinear system,
in cases where the firing rate sequence tends to 2-cycle, i.e $ b<b^* $, can be stated as follows:
{\em
  Let us  consider $ b<0$, 
  an initial condition $ p_0\in X $ 
  (and $N(t)=-\p_vp(V_F)$ for $t\in [-d,0]$)
	and  $p$ its related solution to the nonlinear system
	\eqref{eq:edp-delay}. 
	Assuming that
		the firing rate sequence $\left\{N_{k,\infty}\right\}_{k\ge0}$,
		with initial value $N_{0,\infty}:=-\p_vp_0(V_F)$
		(see \eqref{eq:sequence-N}) tends to a 2-cycle, $\{N^-,N^+\}$,
	then there exist $ d_0>0 $ large enough,
	such that the solution,
	$p$,  to the nonlinear system \eqref{eq:edp-delay} 
	with transmission delay $ d>d_0 $ and initial condition $p_0$
	has the following behaviour:
	\begin{enumerate}
		\item If the initial datum, $p_0$, 
		is $p^-$, the pseudo-equilibrium of the
                nonlinear system \eqref{eq:edp-delay}
                associated to $N^+$ (see \eqref{eq:profile-}),  
		thus, the solution tends
		to a periodic function, $p_\infty^-(v,t)$.
		\item If the initial datum, $p_0$, 
		is $p^+ $,
		the pseudo-equilibrium of the nonlinear system
                \eqref{eq:edp-delay} associated to $N^-$
                (see \eqref{eq:profile+}),
		thus, the solution tends
		to a periodic function, $p_\infty^+(v,t)$.
		\item For a general initial condition $N_0:=-\p_vp_0(V_F)$:
		\begin{enumerate}
			\item If $N_0<N_\infty$, thus,  the system tends to $p_\infty^-(v,t+\delta_{N_0})$.
			\item If $N_\infty<N_0$, thus,  the system tends to $p_\infty^+(v,t+\delta_{N_0})$,
		\end{enumerate}
	with fixed $ \delta_{N_0}\in\R $ for each initial condition $ N_0 $.
		
	\end{enumerate}
}

\section{Conclusions}
\label{sec:conclusions}

In this article we have introduced a discrete system which helps to
better understand the nonlinear leaky integrate and fire (NNLIF) model
when large transmission delay is considered.  This discrete system is
defined only in terms of the system parameters.  It allows us to build
a firing rate and pseudo-equilibria sequences, that determine the
long-time behaviour of the nonlinear system \eqref{eq:edp-delay}.  The
advantage of the discrete model lies in its simplicity.  It allows for
quick simulations that provide accurate information about the NNLIF
system, such as the estimated time to approach equilibrium, whether
the system tends toward a steady state, the possible appearance of
periodic solutions or plateau states, etc.

We have  analytically studied the related discrete system.
Our results give a global view of the asymptotic behaviour of the
discrete system for all possible values
of the connectivity parameter $b$.
Analytically, the link with the nonlinear system \eqref{eq:edp-delay}
has been proved in Theorem \ref{thm:convergence}, but
it only works if $b$ is small enough, in which case the system converges
to its unique equilibrium.
The long-term behaviour for weakly connected networks was already known
using  the entropy dissipation method
\cite{carrillo2014qualitative, caceres2017blow,caceres2019global}.
However, our strategy is different and new, as it describes the behaviour
in relation to the pseudo-equilibria sequence \eqref{eq:sequence-p}.

In addition to our analytical results we show a numerical study,
that leads us to think that the nonlinear system should behave as shown
by the sequence of pseudo-equilibria in all cases with high delay.
The motivation for that conjecture is clear:
large delay means that the nonlinear system is piecewise linear and
with enough time to reach linear equilibria.
The numerical results of this work  describe
the existence of periodic states for the nonlinear system, in the case
of large delay and very negative connectivity parameter, as has been
observed previously by other authors \cite{ikeda2022theoretical}.
We offer range of values for the parameters $b$ and $d$ from
which this phenomenon should occur.
Moreover, we can link the numerical part with our study of the sequence
of pseudo-stationary states, giving a plausible theoretical explanation
for the existence of these oscillations, and shedding some light on
the subsequent analytical test of them, which will be carried out
in future work. We can also relate the study of the sequence
of pseudo-equilibria to the behaviour of the system for $b$ positive.
Finding here a possible explanation for the emergence of the plateau state,
observed in the previous work \cite{CR-L}.

\
          
    To conclude, we summarise the observed global behaviour of systems with large delay in
    terms of its initial firing rate, $N_0=-\p_vp_0(V_F)$:
        \begin{enumerate}
        \item Excitatory networks ($0<b$) show two possibilities: they
          can evolve to:
          \begin{enumerate}
          \item a stationary distribution, the single steady state or the
            steady profile with lower firing rate, in case of two equilibria.
          \item the uniform distribution between $V_R$ and $V_F$
            (plateau state).  This case could occur even with small
            transmission delay value, $d$, if there is no steady state
            ($b$ large), and in cases with two equilibria.
          \end{enumerate}
          \item Inhibitory networks ($b<0$) also show two
            possibilities: they can evolve towards:
          \begin{enumerate}
          \item the stationary distribution, if $b^*<b$.
          \item a periodic solution between the 2-cycle $\{p^-,p^+\}$, if $b<b^*$.
          \end{enumerate}
         \end{enumerate}

\newpage

\appendix

\section{Appendix: Auxiliary calculations}
\label{app:B}
We study the monotonicity of the function
$g(b)=\frac{-\partial_NI(b,N_b^*)}{I(b,N_b^*)^2}$,
used in the proof of Theorem \ref{th:relations-Nb-negative}.
\begin{lem}
  Let us consider $ b\in\left(-\infty,0\right] $
  and $ N(b)$ the solution to the equation $ NI(b,N)=1 $,
  then,  $g(b):=\frac{-\partial_NI(b,N(b))}{I(b,N(b))^2}$
  is an increasing function, defined in $\left(-\infty,0\right] $.
  \label{lem:monotonicity_g}
\end{lem}

\begin{proof}
We consider
$ v(s)\!:=\!s^{-1}e^{-s^2/2}e^{-sbN(b)}\!\!\left(e^{sV_F}\!\!-\!e^{sV_R}\right) $,
$ w(s)\!:=\!e^{-s^2/2}e^{-sbN(b)}\!\!\left(e^{sV_F}\!\!-\!e^{sV_R}\right) $ and $ u(s):=se^{-s^2/2}e^{-sbN(b)}\left(e^{sV_F}-e^{sV_R}\right) $, and rewrite  the function $ I(b,N(b)) $ and some of its useful derivatives as follows:
\begin{equation*}
  I(b,N(b))=
  \int_{0}^{\infty}\!\!\!v(s)ds>0,
\end{equation*}
\begin{equation*}
  \p_N I(b,N(b))=
  -b\!\int_{0}^{\infty}\!\!\!w(s)ds>0,
  \quad
   \p_b I(b,N(b))=
  -N(b)\!\int_{0}^{\infty}\!\!\!w(s)<0,
\end{equation*}
\begin{equation*}
  \p^2_NI(b,N(b))=
  b^2\!\int_{0}^{\infty}\!\!\!\!u(s)ds>0,
  \ \mbox{and} \
  \p_b\p_NI(b,N(b)) =bN(b)\!\int_{0}^{\infty}\!\!\!\!u(s)ds - \int_{0}^{\infty}\!\!\!\!w(s)ds<0,
\end{equation*}
In the following, to shorten the notation we write $ I $ instead of $ I(b,N(b)) $, $ N $ instead of $ N(b) $ and $ N' $ instead of $ \frac{\d N}{\d b}(b) $.
Thus
\begin{equation}\label{eq:d_bg_1}
 g'(b) =\frac{-I\p_b\p_NI-N'I\p_N^2 I+2\p_NI\p_bI+2N'\left(\p_N I\right)^2}{I^3}.
\end{equation}
We derive implicitly in $ I(b,N(b))N(b)=1 $, and obtain
$N'=\frac{-\p_bI}{I^2+\p_N I}$.
        Using it in
        \eqref{eq:d_bg_1}
\begin{equation*}
   g'(b)=\frac{-I^2\p_b\p_NI-\p_NI\p_b\p_NI+
    \p_bI\p_N^2I+2I\p_bI\p_NI}{I^2(I^2+\p_NI)}.
\end{equation*}
The denominator is positive, so we conclude the proof if we prove that the numerator is also positive.
\begin{multline*}
	 g'(b)I^2(I^2+\p_NI)=-\left(\int_{0}^{\infty}v(s)ds\right)^2\left(bN\int_{0}^{\infty}u(s)ds-\int_{0}^{\infty}w(s)ds\right)\\+b\int_{0}^{\infty}w(s)ds\left(bN\int_{0}^{\infty}u(s)ds-\int_{0}^{\infty}w(s)ds\right)\\-Nb^2\int_{0}^{\infty}w(s)ds\int_{0}^{\infty}u(s)ds+2Nb\int_{0}^{\infty}v(s)ds\left(\int_{0}^{\infty}w(s)ds\right)^2\\=-b\int_{0}^{\infty}v(s)ds\int_{0}^{\infty}u(s)ds+\left(\int_{0}^{\infty}v(s)ds\right)^2\int_{0}^{\infty}w(s)ds+b\left(\int_{0}^{\infty}w(s)ds\right)^2,
\end{multline*}
where in the first and last terms we have used the equality $ N=I^{-1}=\left(\int_{0}^{\infty}v(s)ds\right)^{-1} $. It can be easily seen by observing the sign of the different parameters, the only negative term is the last one, so that though the following Cauchy-Schwarz inequality
\begin{equation*}
	\int_{0}^{\infty}v(s)ds\int_{0}^{\infty}u(s)ds\ge\left(\int_{0}^{\infty}w(s)ds\right)^2, \quad\text{as}\quad w(s)=\sqrt{v(s)u(s)},
\end{equation*}
we cancel the last term with the first one and show that $ \frac{\d}{\d b}\p_Nf(b,N(b))>0 \;\forall\; b<0 $.
\end{proof}

\vspace*{8cm}
\subsection*{Acknowledgments}

\thanks{\em The authors acknowledge support from projects of the
  Spanish \emph{Ministerio de Ciencia e Innovación} and the European Regional
  Development Fund (ERDF/FEDER) through grants PID2020-117846GB-I00,
  RED2022-134784-T, and CEX2020-001105-M, all of them funded by
  MCIN/AEI/10.13039/501100011033.}

\newpage
\bibliographystyle{unsrt}
\bibliography{neurons}

\end{document}